\documentclass[12pt]{amsart}
\usepackage{latexsym}
\usepackage{amssymb}
\usepackage{euscript}
\usepackage{graphics}
\usepackage[foot]{amsaddr}
\newcommand{\res}{\mathrm{res}}
\newcommand{\qekv}{\quad \Longleftrightarrow \quad}
\newcommand{\dil}{\mathrm{dil}}
\newcommand{\?}{\; \! \!}
\def\uphar{{\upharpoonright}}
    
\def\sD{{\mathfrak D}}

      \def\dC{{\mathbb C}}
\def\dD{{\mathbb D}}

   \def\dN{{\mathbb N}}   
      
   \def\dT{{\mathbb T}}

\def\cD{{\mathcal D}}      
   \def\cH{{\mathcal H}}   
   \def\cK{{\mathcal K}}   \def\cL{{\mathcal L}}
   \def\cN{{\mathcal N}}   
      
      \def\cU{{\mathcal U}}
      \def\cX{{\mathcal X}}
\def\cY{{\mathcal Y}}

\def\bS{{\mathbf S}}

\def\d1{{\mathcal D}}

\def\bm\chi{\mbox{\boldmath$\chi$}}

\def\ran{{\rm ran\,}}

\let\xker=\ker \def\ker{{\xker\,}}
\def\cspan{{\rm \overline{span}\, }}

\def\d1{{\mathcal D}}

\def\SK{{\bS_{\kappa}( \cU,\cY)}}
\def\So{{\bS( \cU,\cY)}}
\def\T{{\T_{\Sigma}}}

\newtheorem{theorem}{Theorem}[section]
\newtheorem{lemma}[theorem]{Lemma}
\newtheorem{proposition}[theorem]{Proposition}

\numberwithin{equation}{section}
\theoremstyle{definition}
\newtheorem{definition}[theorem]{Definition}
\newtheorem{remark}[theorem]{Remark}
\newtheorem{example}[theorem]{Example}

\begin{document}

\title
{Passive
discrete-time systems with a Pontryagin state space}
\author{L. Lilleberg}
\address{Department of Mathematics and Statistics,
University of Vaasa, P.O. Box 700, 65101 Vaasa,
Finland}
\email{lassi.lilleberg@uva.fi}



\date{}
\begin{abstract}
Passive discrete-time systems 
with  Hilbert spaces 
 as an incoming and outgoing space and a Pontryagin space 
 as a state space are investigated. 
A geometric characterization when the index of the transfer function coincides with the negative index of the state space is given. In this case, an isometric (co-isometric) system has a product representation corresponding to the left (right) Kre\u{\i}n--Langer factorization of the transfer function.   A new criterion, based on the inclusion of reproducing kernel spaces, when a product of  two isometric (co-isometric) systems preserves controllability (observability), is obtained. The concept of the defect function is expanded for generalized Schur functions, and realizations of generalized Schur functions with zero defect functions are studied.
\end{abstract}

\maketitle


\section{Introduction}\label{intro}

Let $\mathcal{U}$ and $\mathcal{Y}$ be separable Hilbert spaces. The \textbf{generalized Schur class} $\bS_{\kappa}( \cU,\cY)$ consists of
 $\mathcal{L}(\cU,\cY)$-valued functions $S(z)$ which are meromorphic in the unit disc $\dD$
 and holomorphic in a neighbourhood $\Omega$ of the origin such that the Schur kernel \begin{equation}\label{kernel1}
K_S(w,z)=\frac{1-S(z)S^*(w)}{1-z\bar{w}}, \qquad w,z \in \Omega,
\end{equation} has $\kappa$ negative squares ($\kappa=0,1,2,\ldots$). This means that for any finite set of points $w_1,\ldots,w_n$
 in the domain of holomorphy $\rho(S)\subset \dD$ of $S$ and vectors $f_1,\ldots,f_n \subset \cY,$ the Hermitian matrix
\begin{equation}\label{kernelmatrix1}
\left(\left\langle K_S(w_j,w_i)f_j,f_i \right\rangle\right)_{i,j=1}^{n}
\end{equation} has at most $\kappa$ negative eigenvalues, and there exists at least one such matrix
 that has exactly $\kappa$ negative eigenvalues. It is known from the reproducing kernel theory  \cite{ADRS}, \cite{Aron}, \cite{rovdrit}, \cite{LangSor},\cite{Sorjonen} that the kernel \eqref{kernel1} generates the reproducing kernel  Pontryagin space $\cH(S)$ with negative index $\kappa.$ The spaces  $\cH(S)$ are called \textbf{generalized de Branges--Rovnyak spaces}, and the elements in $\cH(S)$   are functions defined on $\rho(S)$ with values in $\cY.$   The notation $S^*(z)$ means $(S(z))^*,$ a function ${S}^{\#}(z)$ is defined to be $S^*(\bar z)$ and ${S}^{\#} \in \bS_{\kappa}( \cY,\cU)$ whenever $S \in \SK$ \cite[Theorem 2.5.2]{ADRS}.

 The class $\bS_{0}( \cU,\cY)$ is written as $ \So$ and it coincides with the \textbf{Schur class}, that is, functions holomorphic and bounded by one in $\dD$.  The results first obtained
     by Kre\u{\i}n and Langer \cite{Krein-Langer}, see also {\cite[\S4.2]{ADRS}}  and \cite{DLS1}, show that $S \in \SK$ has \textbf{Kre\u{\i}n--Langer factorizations} of the form $S=S_rB_r^{-1}=B_l^{-1}S_l,$ where $S_r,S_l \in \bS_{0}( \cU,\cY). $ The functions   $B_r^{-1}$ and  $B_l^{-1}$ are inverse Blaschke products, and they have unitary values everywhere on the unit circle $\dT.$ It follows from these factorizations that many properties of the functions in the Schur class  $\So$ hold also for the generalized Schur functions.

  The properties of the generalized Schur functions can be studied by using operator colligations and transfer function realizations.   An \textbf{operator colligation} $\Sigma=(T_{\Sigma};\cX,\cU,\cY;\kappa)$
   consists of a Pontryagin space $\cX$ with the negative index $\kappa$ (\textbf{state space}),
   Hilbert spaces $\cU$ (\textbf{incoming space}), and $\cY$ (\textbf{outgoing space})
   and a \textbf{system operator} $T_{\Sigma}\in \cL(\cX\oplus\cU,\cX\oplus\cY).$ The operator $T_{\Sigma}$
    can be written in the block form
\begin{equation}\label{colli}
T_{\Sigma}=\begin{pmatrix} A&B \cr C&D\end{pmatrix} :
\begin{pmatrix} \cX \\ \cU \end{pmatrix} \to
\begin{pmatrix} \cX \\ \cY \end{pmatrix},
\end{equation} where $A \in \cL(\cX)$ (\textbf{main operator}), $B \in \cL(\cU,\cX)$ (\textbf{control operator}),
 $C \in \cL(\cX,\cY)$ (\textbf{observation
operator}), and $D \in \cL(\cU,\cY)$ (\textbf{feedthrough operator}). Sometimes  the colligation is  written as
 $\Sigma=(A,B,C,D;\cX,\cU,\cY;\kappa).$ It is possible to allow all spaces to be Pontryagin or even Kre\u{\i}n spaces,
 but colligations with only the state space $\cX$ allowed to be a Pontryagin space will be considered in this paper.
 The colligation generated by \eqref{colli} is also called a \textbf{system}
 since it can be seen as a \textbf{linear discrete-time system}
of the form
\begin{equation*}
\label{system}
\begin{cases}
 h_{k+1} &=Ah_k+B\xi_k, \\
  \sigma_{k} &=Ch_k+D\xi_k,
\end{cases}\quad k\geq 0,
\end{equation*}
where $\{h_k\}\subset \cX$, $\{\xi_k\}\subset \cU$ and
$\{\sigma_k\}\subset \cY.$ In what follows, "system" always refers to \eqref{colli}, since other kind of systems are not considered.

  When the system operator $T_{\Sigma}$ in \eqref{colli}
is  a contraction, the corresponding  system is called
\textbf{passive}. If  $T_{\Sigma}$ is isometric
(co-isometric, unitary), then the corresponding system is called isometric
(co-isometric, conservative).  The \textbf{transfer function }of the system \eqref{colli}
is defined by
\begin{equation}
\label{trans} \theta_\Sigma(z):=D+z C(I-z
A)^{-1}B,
\end{equation} whenever $I-z
A$ is invertible. Especially, $\theta$ is defined and holomorphic in a neighbourhood of the origin. The values $\theta_{\Sigma}(z)$ are bounded operators from $\cU$ to $\cY.$
The \textbf{adjoint} or \textbf{dual system} is $\Sigma^*=(T_{\Sigma}^*;\cX,\cY,\cU;\kappa)$ and one has $ \theta_{\Sigma^*}(z)={\theta_{\Sigma}}^\#(z).  $
 Since  contractions between  Pontryagin spaces with the same negative indices are bi-contractions, $\Sigma^*$ is passive whenever $\Sigma$ is.
  If $\theta$ is an $\cL(\cU,\cY)$-valued function and $\theta_{\Sigma}(z)=\theta(z)$ in a neighbourhood of the origin,
then the system $\Sigma$ is called a \textbf{realization} of $\theta.$  A \textbf{realization problem} for the function $\theta \in \SK$ is to find a system $\Sigma$ with a certain minimality property (controllable,  observable, simple, minimal); for details, see Theorem \ref{realz}, such that $\Sigma$ is a realization of $\theta.$

  If $\kappa=0$,
%
%
 the system  reduces to the standard Hilbert space setting of  the
  passive systems studied, for instance, by
de Branges and Rovnyak \cite{BrR1,BrR2},
Ando \cite{Ando},
Sz.-Nagy and Foias \cite{SF},
  Helton \cite{Helton},
   Brodski\u{\i} \cite{Br1},
    Arov \cite{A,Arov} and
     Arov et al \cite{ArKaaP,ArKaaP3,ArNu1,ArNu2,ArSt}.
The theory has been extended to Pontryagin state space case  by Dijksma et al \cite{DLS1,DLS2}, Saprikin \cite{Saprikin1},
Saprikin and Arov  \cite{SapAr} and Saprikin et al  \cite{Saprikinetal}. Especially, in \cite{Saprikin1}, Arov's well-known results of minimal and optimal minimal systems are generalized to the Pontryagin state space settings. Part of those results  are used in \cite{Saprikinetal}, where transfer functions, Kre\u{\i}n--Langer factorizations, and the corresponding product representation of system are studied and, moreover, the connection between bi-inner transfer functions  and systems with bi-stable main operators are generalized to the Pontryagin state space settings. In this paper those results will be further expanded and improved.

 The case when all the spaces are indefinite, the theory of isometric, co-isometric and conservative systems is  considered, for instance,  in \cite{ADRS}, see also \cite{rovdrit}. 
 The indefinite reproducing kernel spaces 
 were first studied by Schwartz in  \cite{Schwartz} and Sorjonen in \cite{Sorjonen}.

The paper is organized as follows.
In Section \ref{sec-pont}, basic notations and definitions about the indefinite spaces and their operators are given. Also,  the  left and right Kre\u{\i}n--Langer factorizations are formulated, and the  boundary value properties of generalized Schur functions are introduced.  After that, basic properties  of linear discrete time systems, or operator colligations, especially in Pontryagin state space, are recalled without proofs. However, the extension of Arov's result about the weak similarity between two minimal passive realizations of the same transfer function,
is given with a proof.

 Section \ref{sec-products} deals mainly with the dilations, embeddings and products of two systems.
  The transfer function $\theta_\Sigma$ of the passive system $\Sigma=(T_{\Sigma};\cX,\cU,\cY;\kappa)$ is a generalized Schur function with negative index no larger than the negative index of the state space $\cX$, but the theory of passive systems will often be meaningful only if the indices are equal. A simple geometric criterion for these indices to coincides is given in Lemma \ref{simp-kar}. Main results in this section contain
   criteria when the product of two co-isometric  (isometric) systems preserves observability (controllability). These results are obtained  in Theorems \ref{obspre} and \ref{contpre}). The criteria involve the reproducing kernel spaces induced by the generalized Schur functions.  Moreover, Theorem \ref{reps1} expands the results of \cite{Saprikinetal} about the realizations of generalized Schur functions and their product representations  corresponding to the Kre\u{\i}n--Langer factorizations. In the end of Section \ref{sec-products}, it is obtained  that if $A$ is the main operator of $\Sigma=(T_{\Sigma};\cX,\cU,\cY;\kappa)$ such that $\theta_\Sigma \in \SK,$ then there exist unique fundamental decompositions
 $\cX= \cX_1^+ \oplus \cX_1^-= \cX_2^{+} \oplus \cX_2^{-}$ such that $A\cX_1^+ \subset \cX_1^+$ and $A\cX_2^{-}\subset \cX_2^{-},$ respectively; see Proposition \ref{invariants}.

   Section \ref{sec-defect} expands and   generalizes the results of \cite{Arov} and \cite{Saprikinetal} about the realizations of bi-inner functions. It will be shown that the notions of stability and co-stability can be generalized to the Pontryagin state space settings in  a similar manner as bi-stability is generalized in \cite{Saprikinetal}. Moreover, the results of \cite{AHS2} about the realizations of ordinary Schur functions with zero defect functions will be generalized. This yields a class of generalized Schur functions with boundary value properties very close to those of inner functions in a certain sense.

\section{Pontryagin spaces, Kre\u{\i}n-Langer factorizations and linear systems}\label{sec-pont}
Let $\cX$ be a complex vector space  with a Hermitian indefinite inner product $\langle \cdot , \cdot \rangle_{\cX}.$ The anti-space of $\cX$ is the space $-\cX$ that coinsides with  $\cX$ as a vector  space but its inner product is $- \langle \cdot , \cdot \rangle_{\cX}.$ Notions of orthogonality and orthogonal direct sum are defined as in the case of Hilbert spaces, and $\cX \oplus \cY$ is often denoted by $\begin{pmatrix}
                                                                                                                                                                                     \cX &
                                                                                                                                                                                     \cY
                                                                                                                                                                                   \end{pmatrix}^\intercal.$
                                                                                                                                                                                   Space $\cX$ is said to be a \textbf{Kre\u{\i}n space} if it admits a decomposition
$\cX=\cX^{+} \oplus \cX^{-}$ where $(\cX^{\pm},\pm \langle \cdot , \cdot \rangle_{\cX})$ are Hilbert spaces. Such a  decomposition is called a \textbf{fundamental decomposition}. In general, it is not unique. However, a fundamental decomposition determines the Hilbert space $|\cX |=\cX^{+} \oplus \left(-\cX^{-}\right)$ with the strong topology which does not depend on the choice of the fundamental decomposition.
 The dimensions of $\cX^+$ and $\cX^-,$ which are also independent of the choice of the fundamental decomposition,  are called the \textbf{positive }and  \textbf{negative indices} $\mathrm{ind}_{\pm}\,\cX= \mathrm{dim} \,\cX^{\pm}
$ of $\cX.$   In what follows, all notions of continuity and convergence are understood
to be with respect to the strong topology. All spaces are assumed to be separable.
A linear manifold $\cN \subset \cX$ is  a \textbf{regular subspace}, if it is itself a Kre\u{\i}n space with the inherited   inner product of $\langle \cdot , \cdot \rangle_{\cX}.$ A \textbf{Hilbert subspace} is a regular subspace such that its negative index is zero, and a uniformly negative subspace is a regular subspace with positive index zero, i.e., an \textbf{anti-Hilbert space.}
 If $\cN \subset \cX$ is a regular subspace, then  $\cX=\cN \oplus \cN^\perp,$ where $\perp$ refers to orthogonality w.r.t.  
  indefinite inner product $\langle \cdot , \cdot \rangle_{\cX}.$ Observe that $\cN$ is regular precisely when $\cN^\perp$ is regular.

 Denote by  $\cL(\cX,\cY)$ the space of all continuous linear operators from the Kre\u{\i}n space $\cX$ to the Kre\u{\i}n space $\cY.$ Moreover,
 $\cL(\cX)$ stands for  $\cL(\cX,\cX)$. Domain of a linear operator $T$ is denoted by $\cD(T),$ kernel by $\ker T$ and $T\uphar_{\cN}$ is a restriction of $T$ to the linear manifold $\cN.$ The   adjoint of $A\in \cL(\cX,\cY)$ is an operator $A^* \in \cL(\cY,\cX) $ such that $\langle Ax,y \rangle_{\cY} =\langle x,A^*y \rangle_{\cX} $ for all $x \in \cX$ and $y \in \cY.$ Classes of invertible, self-adjoint, isometric, co-isometric and unitary operators are defined as for Hilbert spaces, but with respect to the indefinite inner product. For self-adjoint  operators $A, B\in \cL(\cX,\cY), $ the inequality $A\leq B$ means  that $\langle Ax,x \rangle \leq \langle Bx,x  \rangle$ for all $x \in \cX$.  A self-adjoint operator $P \in\cL(\cX)$ is an $\langle \cdot , \cdot \rangle$-\textbf{orthogonal  projection} if  $P^2=P$. The unique orthogonal projection onto a regular subspace $\cN$ of $\cX$ exists and is denoted by $P_{\cN}$.
 A \textbf{ Pontryagin space} is a Kre\u{\i}n space $\cX$ such that $\mathrm{ind}_{-}\,\cX<\infty.$  A linear operator $A\in \cL(\cX,\cY)$  is a \textbf{contraction} if $\langle Ax,Ax\rangle \leq \langle x,x\rangle$ for all $x \in \cX.$ If $\cX$ and $\cY$ are Pontryagin spaces with the same negative index, then the adjoint of a contraction $A\in \cL(\cX,\cY)$ is still a contraction, i.e., $A$ is a \textbf{bi-contraction}. The identity operator of the space $\cX$ is denoted by $I_{\cX}$ or just by $I$ when the corresponding space is clear from the context. For further information about the indefinite spaces and their operators, we refer to \cite{Azizov}, \cite{Bognar} and \cite{rovdrit}.

  For  ordinary Schur class $ \So,$ 
  it is well known \cite{SF} that $S \in \So $ has non-tangential strong limit values  almost everywhere (a.e.) on the unit circle $\dT.$
   It follows that $S \in \So$ can be extended to
   $L^{\infty}(\cU,\cY)$ function, that is,
    the class of weakly measurable a.e. defined and essentially bounded $\mathcal{L}(\cU,\cY)$-valued functions on $\dT.$
     Moreover, $S(\zeta)$ is contractive a.e. on $\dT.$ If $S \in \So$ has isometric (co-isometric, unitary) boundary values a.e. on $\dT,$ then $S$ is said to be \textbf{inner} (\textbf{co-inner}, \textbf{bi-inner}).

 If $\cU=\cY,$ then the notations $\mathbf{S}(\cU)$ and $\mathbf{S}_{\kappa}(\cU)$ are often used instead of $\mathbf{S}(\cU,\cU)$ and $\mathbf{S}_{\kappa}(\cU,\cU)$.
   Suppose that $P \in \cL(\cU)$ is an orthogonal projection from the Hilbert space $\cU$ to an arbitrary one dimensional subspace. Then a function defined by
  \begin{equation}\label{bla-fo-fact}
    b(z)=I-P +\rho \,\frac{z-\alpha}{1-\bar \alpha z}\,P, \qquad |\rho| =1,  \qquad 0<|\alpha|<1,
  \end{equation} is a \textbf{simple Blaschke-Potapov factor.}
  Easy calculations show that $b$ is holomorphic in the closed unit disc $\overline{\dD},$
  it has unitary values everywhere on $\dT$ and $b(z)$ is invertible whenever $z \in \overline \dD \setminus \{\alpha\}.$
In particular,  $b \in \mathbf{S}_0(\mathcal U)$ 
 is bi-inner.
  A finite product
  \begin{equation}\label{bla-pro}
  B(z)= \prod_{k=1}^{n}\left( I-P_k +\rho_k \frac{z-\alpha_k}{1-\bar \alpha_k z}P_k\right),  \qquad |\rho_{k}| =1,  \qquad 0<|\alpha_{k}|<1,
  \end{equation} of simple  Blaschke-Potapov factors  is called \textbf{ Blaschke product of degree $n$,}
   and it is also bi-inner and invertible on $\overline \dD \setminus \{\alpha_1,\ldots, \alpha_n\}.$
 The  following factorization theorem was first obtained
     by Kre\u{\i}n and Langer \cite{Krein-Langer}, see also {\cite[\S4.2]{ADRS}}  and \cite{DLS1}.
  \begin{theorem}\label{krein-langer-fact}
    Suppose $S \in \SK.$ Then
    \begin{equation}\label{right-krein}
      S(z)=S_r(z)B^{-1}_r(z)
    \end{equation} where $S_r \in \So$ and $B_r$ is a Blaschke product of degree $\kappa$
     with values in $\mathcal{L}(\mathcal U) $ such that $B_r(w)f=0$ and $S_r(w)f=0$ for some $w \in \dD$ only if $f=0.$
     Moreover,
    \begin{equation}\label{le-kre}
      S(z)=B^{-1}_l(z)S_l(z)
    \end{equation} where $S_l \in \So$ and $B_l$ is a Blaschke product of degree $\kappa$
     with values in $\mathcal{L}(\mathcal Y)$ such that $B_l^*(w)g=0$ and $S_l^*(w)g=0$ for some $w \in \dD$
    only if $g=0.$

    Conversely, any function of the form \eqref{right-krein}
      or \eqref{le-kre} belongs to $\mathbf{S}_{\kappa'}$ for some $\kappa'\leq \kappa, $
       and $\kappa'=\kappa$ exactly when the functions have no common zeros in sense as  described above.
        Both factorizations are unique up to unitary constant factors.
  \end{theorem}  The factorization \eqref{right-krein} is called the \textbf{right Kre\u{\i}n-Langer factorization}
   and \eqref{le-kre} is the \textbf{left Kre\u{\i}n-Langer factorization.}
  It follows that $S \in \SK$ has $\kappa$ poles (counting multiplicities) in $\dD,$
    contractive strong limit values exist a.e. on $\dT$ and $S$ can also be extended to  $L^{\infty}(\cU,\cY)$-function.
    Actually, these properties also characterize the generalized Schur functions. This result will be stated for reference purposes. For the proof of the sufficiency, see   \cite[Proposition 7.11]{DLS1}.
    \begin{lemma}\label{kar-gen}
    Let  $S$ be  an $\cL(\cU,\cY)$-valued function  holomorphic at the origin. Then  $S \in \SK$ if and only if $S$  is meromorphic  on $\dD$ with finite pole multiplicity $\kappa$ and
        $$  \lim_{r \to 1} \sup_{ |z|=r} \| S(z)\|\leq1      $$ holds.
    \end{lemma}

    A function $S \in \SK$ 
     and the factors $S_r$ and $S_l$  in \eqref{right-krein} and  \eqref{le-kre}  have simultaneously isometric (co-isometric, unitary) boundary values since the factors $B^{-1}_l$ and $B^{-1}_r$ have unitary values everywhere on $\dT.$

  The following result \cite[Theorem V.4.2]{SF}, which involves the notion of an \textbf{outer function} (for the definition, see \cite{SF}), will be utilized.
   \begin{theorem}\label{majorant}
If $\cU$ is a separable Hilbert space and $N\in L^{\infty}(\cU)$ such that $0 \leq N(\zeta)\leq I_{\cU}$ a.e. on $\dT,$
   then there exist a Hilbert space $\cK$ and an outer function $\varphi \in \mathbf{S}(\cU,\cK)$ such that
   \begin{itemize}
     \item[\rm(i)] $\varphi^*(\zeta)\varphi(\zeta)\leq N^2(\zeta)$ a.e. on $\dT;$
     \item[\rm(ii)] if $\widehat{\cK}$ is a Hilbert space and $\widehat{\varphi}\in \mathbf{S}(\cU,\widehat{\cK})$ such that
     $\widehat{\varphi}^*(\zeta)\widehat{\varphi}(\zeta)\leq N^2(\zeta)$ a.e. on $\dT,$ then
     $\widehat{\varphi}^*(\zeta)\widehat{\varphi}(\zeta)\leq \varphi^*(\zeta)\varphi(\zeta)$ a.e. on $\dT.$
   \end{itemize}
   Moreover, $\varphi$ is unique up to a left constant unitary factor.
   \end{theorem}

   For $S \in \SK$ with the Kre\u{\i}n--Langer factorizations $S=S_rB_r^{-1}=B_{l}^{-1}S_l,$ define
    \begin{align*} N_{S}^2(\zeta)&:=I_{\cU}-S^*(\zeta)S(\zeta), \qquad  \mbox{a.e.} \quad\zeta \in \dT, \\
    M_{S}^2(\zeta)&:=I_{\cY}-S(\zeta)S^*(\zeta), \qquad   \mbox{a.e.} \quad\zeta \in \dT. \end{align*}
      Since Blaschke products are unitary on $\dT,$ it follows that
           \begin{align}N_{S}^2(\zeta) &= I_{\cU}-S_l^*(\zeta)S_l(\zeta)  \,=N_{S_l}^2(\zeta) \label{leftde}
           \\M_S^2(\zeta)&= I_{\cY}-S_r(\zeta)S_r^*(\zeta)  =M_{S_r}^2(\zeta)\label{rightde}.\end{align}
Theorem \ref{majorant} guarantees that there exists an outer function  $\varphi_{S}$ with properties introduced in Theorem \ref{majorant}  for $N_{S}.$   An easy modification of Theorem \ref{majorant} shows that there exists a Schur function $\psi_S$ such that $\psi_S^\#$ is an outer function, $ \psi_S(\zeta) \psi_S^*(\zeta) \leq  M_S^2(\zeta)$  a.e. $\zeta \in\dT$ and $ \psi_S(\zeta) \psi_S^*(\zeta) \leq  \widehat{\psi}(\zeta) \widehat{\psi}^*(\zeta)$ for every Schur function $\widehat{\psi}$ with a property $ \widehat{\psi}_S(\zeta) \widehat{\psi}_S^*(\zeta) \leq  M_S^2(\zeta).$  Moreover,it follows from the identies \eqref{leftde} and \eqref{rightde} that
\begin{align}\varphi_{S} =\varphi_{S_l}\quad\text{and}\quad
           \psi_S= \psi_{S_r} \label{rightdefect}.\end{align}

     The     function $\varphi_S$ is called the  \textbf{right defect function} and $\psi_S$  is  the \textbf{left defect function}.

Let $\Sigma=(A,B,C,D;\cX,\cU,\cY;\kappa)$ be a passive system.
The following subspaces \begin{align}
\cX^c&:=\cspan\{\ran A^{n}B:\,n=0,1,\ldots\} \label{cont1},  \\
\cX^o&:= \cspan\{\ran A^{*n}C^*:\,n=0,1,\ldots\}\label{obs1},  \\
\cX^{s}&:= \cspan\{\ran A^{n}B,\ran A^{*m}C^*:\,n,m=0,1,\ldots\},  \label{simp1}
\end{align}
 are called respectively controllable,  observable and simple subspaces.
 The system $\Sigma$ is said to be \textbf{controllable} (\textbf{observable}, \textbf{simple})
  if $\cX^c=\cX (\cX^o=\cX,\cX^s=\cX)$ and \textbf{minimal} if it is both controllable and observable.
   When $\Omega\ni0$ is some symmetric neighbourhood of the origin, that is, $\bar{z} \in \Omega$ whenever $z\in \Omega,$ then also
 \begin{align}
\cX^c&=\cspan\{\ran (I-zA)^{-1}B: z \in \Omega\} \label{cont2},  \\
\cX^o&=\cspan\{\ran (I-zA^*)^{-1}C^*: z \in \Omega\}\label{obs2},  \\
\cX^{s}&=\cspan\{\ran (I-zA)^{-1}B,\ran (I-wA^*)^{-1}C^*: z,w \in \Omega\} \label{simp2}.
\end{align}
 If the system operator $T_\Sigma$ in \eqref{colli} is a contraction, that is, $\Sigma$ is passive,  the operators
 $$A: \cX \to \cX, \quad  \begin{pmatrix}
                    A \\
                    C
                  \end{pmatrix}: \cX \to   \begin{pmatrix}
                    \cX \\
                    \cY
                  \end{pmatrix} , \quad \begin{pmatrix}
                    A &
                    B
                  \end{pmatrix}: \begin{pmatrix}
                    \cX \\
                    \cU
                  \end{pmatrix} \to \cX,      $$ are also bi-contractions. Moreover, the
 operators $B$ and $C^*$ are contractions but not bi-contractions unless $\kappa=0.$

The following realization theorem is known, and the parts (i)--(iii) can be found e.g. in  \cite[Chapter 2]{ADRS} and the part (iv) in \cite[Theorem 2.3 and Proposition 3.3]{Saprikin1}. 
\begin{theorem}\label{realz}
  For $\theta \in \SK$ there exist realizations $\Sigma_k, k=1,\ldots,4,$ of $\theta$ such that
  \begin{itemize}
    \item[\rm{(i)}] $\Sigma_1$ is conservative and simple;
    \item[\rm{(ii)}] $\Sigma_2$ is  isometric and controllable;
    \item[\rm{(iii)}]$\Sigma_3$ is  co-isometric and observable;
    \item[\rm{(iv)}]  $\Sigma_4$ is  passive and minimal.
  \end{itemize} Conversely, if the system $\Sigma$ has some of the properties {\rm(i)--(iv)},
   then $\theta_{\Sigma}\in \SK,$ where $\kappa$ is the negative index of the state space of $\Sigma.$
 \end{theorem}
It is also true that the transfer function of passive system is a generalized Schur function,
 but its index may be smaller than the  negative index of the state space \cite[Theorem 2.2]{Saprikin1}. For  a conservative system $\Sigma$ it is known from  \cite[Theorem 2.1.2 (3)]{ADRS} that the index of the transfer function $\theta_\Sigma$ of $\Sigma$ co-insides with the negative index of the state space $\cX$ of $\Sigma$  if and only if the space $
(\cX^{s})^\perp$ is a Hilbert subspace. This result holds also in more general settings when $\Sigma$ is passive, as it will be proved in Lemma \ref{simp-kar}, after introducing  some machinery.

Two  realizations $\Sigma_1=(A_1,B_1,C_1,D_1;\cX_1,\cU,\cY;\kappa)$ and  $\Sigma_2=(A_2,B_2,C_2,D_2;\cX_2,\cU,\cY;\kappa)$
 of the same function $\theta \in \SK$ are  called \textbf{unitarily similar} \label{unsims} if $D_1=D_2$ and there exists a unitary operator $U: \cX_1 \to \cX_2 $
  such that $$  A_1=U^{-1}A_2U, \quad B_1=U^{-1}B_2, \quad C_1=C_2U.   $$ Moreover, the realizations $\Sigma_1$ and $\Sigma_2$ are said to be \textbf{weakly similar} if  $D_1=D_2$
  and there exists an injective closed densely defined possibly unbounded linear operator $Z: \cX_1 \to \cX_2$ with the dense range such that
 \begin{equation*}\label{weaksim}
ZA_1f=A_2Zf, \quad C_1f=C_2Zf, \quad  f\in\cD( Z),\quad  \mbox{and} \quad
ZB_1=B_2.      \end{equation*}

 Unitary similarity preserves dynamical properties of the system and also the spectral properties of the main operator.
  If two realizations of $\theta \in \SK$ both have the same property {\rm(i), (ii)} or { \rm(iii)} of Theorem \ref{realz},
   then they are unitarily similar \cite[Theorem 2.1.3]{ADRS}. In Hilbert state space case,
    results of Helton \cite{Helton} and Arov \cite{A} state that two minimal passive realizations of
    $\theta \in \So$ are weakly similar. However, weak similarity preserves neither
     the dynamical properties of the system nor the spectral properties of its main operator.
      The following theorem shows that Helton's and Arov's statement holds also in Pontryagin state space settings.
       Proof is similar to the one given in the Hilbert space settings  in \cite[Theorem 3.2]{BallCohen} and \cite[Theorem 7.13]{BGKR}.

 \begin{theorem}\label{weaks}
    Let $\Sigma_1=(
    T_{\Sigma_1};\cX_1,\cU,\cY;\kappa)$ and  $\Sigma_2=( T_{\Sigma_2}
    ;\cX_2,\cU,\cY;\kappa)$
     be two minimal passive realizations of $\theta \in \SK.$ Then they are weakly similar.
 \end{theorem} \begin{proof}
                 Decompose the system operators as in \eqref{colli}. In a sufficiently small neighbourhood of the origin, 
                  the functions $\theta_{\Sigma_1}$ and $\theta_{\Sigma_2}$
                   have the Neumann series which coincide.  Hence $D_1=D_2$ and $C_1A_1^kB_1=C_2A_2^kB_2$ for any  $k \in \dN_0=\{0,1,2,\ldots\}.$ Since $\Sigma_1$ is controllable,
                    vectors of the form $x=\sum_{k=0}^N A_1^k B_1 u_k, u_k \in \cU,$ are dense in $\cX_1.$ 
                     Define $$Rx=\sum_{k=0}^N A_2^k B_2 u_k,$$ and let $Z$ be the closure of the graph of $R.$      Let $\{x_n\}_{n\in \dN}\subset \mathrm{span}\{\ran A_1^{k}B_1:\,k \in \dN_0\} =\cD(R)$ such that $x_n \to 0$ and $Rx_n \to y$ when $n\to \infty.$ Since  $C_1A_1^kB_1=C_2A_2^kB_2$ for any $k\in \dN_0,$ also $C_1A_1^kx_n=C_2A_2^kRx_n,$ and the continuity implies
                    $$
                    C_2A_2^ky =\lim_{n \to \infty }  C_2A_2^kRx_n = \lim_{n \to \infty }C_1A_1^kx_n=0. $$
                  Since $\Sigma_2$ is observable, it follows from \eqref{obs1} that
                 \begin{equation} \label{zeroker}
                 \bigcap_{k \in \dN_0 } \ker  C_2A_2^k =\{0\},  \end{equation}
                 and therefore $y=0.$
                   This implies that $Z$ is a  closed densely defined linear operator. Since $\Sigma_2$ is controllable, the range of $Z$ is dense. 

                 To prove the injectivity,
                     let $x \in \cD(Z)$ such that $Zx=0.$ Then there exists
                  $\{x_n\}_{n\in \dN}\subset \cD(R)
                  $ such that $x_n \to x$ and $Rx_n \to 0.$
                %
                   By the continuity, $$C_1A_1^kx=\lim_{n \to \infty }C_1A_1^kx_n =\lim_{n \to \infty }  C_2A_2^kRx_n=0$$ for any $k \in \dN_0.$ Since $\Sigma_1$ is observable, this implies that $x=0,$ and $Z$ is injective.

                   For $x \in \cD(Z),$ there exists  $\{x_k\}_{k\in \dN}\subset \cD(R)$ such that $x_k \to x$ and $Rx_k \to Zx.$ Then    \begin{align}   A_1x &=\lim_{k \to \infty } A_1x_k\label{eq1} \\                 
                   A_2Zx &= \lim_{k \to \infty } A_2Rx_k =  \lim_{k \to \infty }RA_1x_k =   \lim_{k \to \infty }  ZA_1x_k \label{eq2}\\
                C_1x &=  \lim_{k \to \infty }  C_1x_k=  \lim_{k \to \infty } C_2Rx_k=C_2Zx \label{eq3} \\
                ZB_1&=RB_1=B_2.  \label{eq4}
                  \end{align} Since $Z$ is closed, equations \eqref{eq1}  and \eqref{eq2} show that  $A_1 x \in \cD(Z)$ and $ZA_1 x=A_2 Zx.$ Since \eqref{eq3} and \eqref{eq4} hold also, it has been shown that $Z$ is a weak similarity.
               \end{proof}
       \begin{remark}  It should be noted that Theorem \ref{weaks} holds also when all the spaces are Pontryagin,  Kre\u{\i}n or, if one defines the observability criterion as $\cap_{n \in \dN_0}\ker CA^n=\{0\},$ even Banach spaces. This result can also be derived from \cite[p. 704]{Staffans}.
\end{remark}

          \section{Julia operators, dilations, embeddings and products of systems}\label{sec-products}
         The system \eqref{colli} can be expanded to a larger system  either without changing the transfer function
          or without changing the main operator. Both of these can be done by using the \textbf{Julia operator}, see \eqref{Juliaoperator} below.
          For a proof of the next theorem  and some further details about Julia operators, see \cite[Lecture 2]{rovdrit}.

          \begin{theorem}\label{Julia} Suppose that $\cX_1$  and $\cX_2$ are Pontryagin spaces with the same negative index,
           and $A: \cX_1 \to \cX_2$ is a contraction. Then there exist Hilbert spaces $\sD_{A}$
            and $\sD_{A^*},$ linear operators $D_A: \sD_{A}\to\cX_1, D_{A^*}: \sD_{A^*} \to \cX_2 $
             with zero kernels  and a linear operator $L:\sD_{A} \to \sD_{A^*}  $ such that \begin{equation}\label{Juliaoperator}
                       U_A:= \begin{pmatrix}
                               A & D_{A^*} \\
                                D^*_{A} & -L^*
                             \end{pmatrix}: \begin{pmatrix}
                                              \cX_1 \\
                                              \sD_{A^*}
                                            \end{pmatrix} \to \begin{pmatrix}
                                             \cX_2 \\
                                              \sD_{A}
                                            \end{pmatrix}
                     \end{equation} is unitary. Moreover, $U_A$ is essentially unique.
          \end{theorem}

         A \textbf{ dilation} of a system \eqref{colli} is
           any system of the form $\widehat{\Sigma}=(\widehat{A},\widehat{B},\widehat{C},D;\widehat{\cX},\cU,\cY;\kappa'),$ where
  \begin{equation}\label{dilation}
 \widehat{\cX}=  \cD\oplus{\cX}\oplus\cD_*  , \quad \widehat{A} \cD \subset \cD, \quad \widehat{A}^* \cD_* \subset \cD_*, \quad \widehat{C} \cD =\{0\}, \quad \widehat{B}^* \cD_* =\{0\}.
 \end{equation} That is, the system operator $T_{\widehat{\Sigma}}$ of $\widehat{\Sigma}$ is of the form
  \begin{equation}\label{dilatio-blok}\begin{split}
    T_{\widehat{\Sigma}}=\begin{pmatrix}
                           \begin{pmatrix}
                             A_{11} & A_{12} & A_{13} \\
                             0 & A & A_{23} \\
                             0 & 0 & A_{33}
                           \end{pmatrix} & \begin{pmatrix}
                                             B_1\\
                                            B \\
                                            0
                                          \end{pmatrix} \\
                           \begin{pmatrix}
                                             0&
                                            C &
                                            C_1
                                          \end{pmatrix} & D
                         \end{pmatrix}: \begin{pmatrix}
                                          \begin{pmatrix}
                                            \cD \\
                                            \cX \\
                                            \cD_*
                                          \end{pmatrix} \\
                                          \cU
                                        \end{pmatrix} \to\begin{pmatrix}
                                          \begin{pmatrix}
                                            \cD \\
                                            \cX \\
                                            \cD_*
                                          \end{pmatrix} \\
                                          \cY
                                        \end{pmatrix}, \\
                                        \widehat{A}=\begin{pmatrix}
                             A_{11} & A_{12} & A_{13} \\
                             0 & A & A_{23} \\
                             0 & 0 & A_{33}
                           \end{pmatrix} ,\qquad\widehat{B}=\begin{pmatrix}
                                             B_1\\
                                            B \\
                                            0
                                          \end{pmatrix},\qquad\widehat{C}= \begin{pmatrix}
                                             0&
                                            C &
                                            C_1
                                          \end{pmatrix}. \end{split}
  \end{equation}

  Then the system $\Sigma$ is called a \textbf{restriction} of $\widehat{\Sigma},$ and it has an expression \begin{equation}\label{restri} \Sigma=(P_{\cX}\widehat{A}\uphar_{\cX},P_{\cX}\widehat{B},\widehat{C}\uphar_{\cX},D;P_{\cX}\widehat{\cX},\cU,\cY;\kappa).\end{equation}
  Dilations and restrictions are denoted by
  \begin{equation}\label{dil-res}
 \widehat{\Sigma}=\dil_{\cX \to \widehat{\cX}} \Sigma, \quad \Sigma=\res_{ \widehat{\cX} \to\cX} \widehat{\Sigma},
 \end{equation} mostly without subscripts when the corresponding state spaces are clear.
  A calculation show that the transfer functions of the original system and its dilation coincide.

 The second way to expand the  system \eqref{colli} is called an  \textbf{embedding}, which is
  any system determined by the system operator 
 \begin{equation}\label{embed}
   T_{\widetilde{\Sigma}}\?=\begin{pmatrix}
                             A & \widetilde{B} \\
                            \widetilde{C} & \widetilde{D}
                           \end{pmatrix}: \begin{pmatrix}
                              \cX     \\  \widetilde{\cU}
                                 \end{pmatrix} \to \begin{pmatrix}
                                 \cX \\  \widetilde{\cY}
                                 \end{pmatrix} \! \qekv  \! \begin{pmatrix}
                             A & \begin{pmatrix}
                                   B & B_1
                                 \end{pmatrix} \\
                           \begin{pmatrix}
                                   C \\ C_1
                                 \end{pmatrix} & \begin{pmatrix}
                                                   D & D_{12} \\
                                                   D_{21} & D_{22}
                                                 \end{pmatrix}
                           \end{pmatrix}: \begin{pmatrix}
                                            \cX \\
                                            \begin{pmatrix}
                                              \cU \\
                                              \cU'
                                            \end{pmatrix}
                                          \end{pmatrix}\to \begin{pmatrix}
                                            \cX \\
                                             \begin{pmatrix}
                                              \cY \\
                                              \cY'
                                            \end{pmatrix}
                                          \end{pmatrix},
 \end{equation}  where $\cU'$ and $\cY'$ are Hilbert spaces.
 The transfer function of the embedded system is \begin{equation}\label{transemb}\begin{split}
                                                             \theta_{\widetilde{\Sigma}}(z)&= \begin{pmatrix}
                                                   D & D_{12} \\
                                                   D_{21} & D_{22}
                                                 \end{pmatrix} +z  \begin{pmatrix}
                                   C \\ C_1
                                 \end{pmatrix}(I_{\cX}-zA)^{-1}\begin{pmatrix}
                                   B & B_1
                                 \end{pmatrix} \\ &= \begin{pmatrix}
                                                       D + zC(I_{\cX}-zA)^{-1}B & D_{12}+zC(I_{\cX}-zA)^{-1}B_1 \\
                                                          D_{21} +zC_1(I_{\cX}-zA)^{-1}B & D_{22} + zC_1(I_{\cX}-zA)^{-1}B_1
                                                     \end{pmatrix} =
                                                             \begin{pmatrix}
                                                                                           \theta_{\Sigma} (z) & \theta_{12}(z)\\
                                                                                            \theta_{21}(z) &  \theta_{22}(z)
                                                                                         \end{pmatrix},
                                                           \end{split}\end{equation}
                                                            where $\theta_{\Sigma}$ is the transfer function of the original system.


  For a passive system there always exist
 a conservative dilation \cite[Theorem 2.1]{Saprikin1} and a conservative embedding \cite[p. 7]{Saprikinetal}.
  Both of these can be constructed such that the system operator of the expanded system is the Julia operator of $T_{\Sigma}$. 
   Such expanded systems are called \textbf{Julia dilation} and \textbf{Julia embedding}, respectively.

   If the passive system \eqref{colli} is simple (controllable, observable, minimal), then so is any conservative embedding \eqref{embed} of it. This follows from the fact that $B\cU \subset \widetilde{B}\widetilde{\cU} $  and $C^*\cY\subset \widetilde{C}^*\widetilde{\cY}.$
   A detailed proof of simplicity can be found  in \cite[Theorem 4.3]{Saprikinetal}. The same argument works also  in the rest of the cases.
    However, it can happen that a simple passive system has no  simple conservative dilation, even in the case when the original system is minimal, see  the example on page 15 in \cite{Saprikinetal}.

\begin{lemma}\label{simp-kar}
  Let $\theta_{\Sigma}$ be the transfer function of a  passive system  $\Sigma=(T_\Sigma;\cX,\cU,\cY;\kappa).$ If $\theta_\Sigma \in \SK,$ then  the spaces $(\cX^c)^\perp,(\cX^o)^\perp$ and $ (\cX^s)^\perp$ are Hilbert subspaces of $\cX.$ Moreover, if one of the spaces $(\cX^c)^\perp,(\cX^s)^\perp$ and $ (\cX^s)^\perp$ is a Hilbert subspace, then so are the others and $\theta_\Sigma \in \SK.$  \end{lemma}
\begin{proof} 
 If $\theta_{\Sigma}\in \SK,$ it is proved in \cite[Lemma 2.5]{Saprikin1} that $(\cX^c)^\perp$ and $(\cX^o)^\perp$   are Hilbert spaces. It easily follows from \eqref{cont1} and \eqref{obs1} that \begin{equation}\label{intse}(\cX^s)^\perp=(\cX^c)^\perp\cap(\cX^o)^\perp,\end{equation} so $(\cX^s)^\perp$ is also a Hilbert space, and the first claim is proved.

 Suppose next  that $(\cX^s)^\perp$   is a Hilbert space. Consider a conservative embedding $\widetilde{\Sigma}$ of $\Sigma,$   and represent the system operator $T_{\widetilde{\Sigma}}$ as in \eqref{embed}.
  The first identity in \eqref{transemb} shows that the transfer function of any embedding of $\Sigma$ has the same number of poles (counting multiplicities)   as $\theta_\Sigma,$ and therefore it follows from 
   Lemma \ref{kar-gen} that the indices of  $\theta_{\Sigma}  $ and $\theta_{\widetilde{\Sigma}}$ coincides. Denote the simple subspace of the embedded system as $\widetilde{\cX}^s.$ Since ${\cX}^s\subset\widetilde{\cX}^s,$ it holds $(\widetilde{\cX}^s)^\perp \subset (\cX^s)^\perp,$ and therefore $(\widetilde{\cX}^s)^\perp$ is also a Hilbert space. It follows  from  \cite[Theorem 2.1.2 (3)]{ADRS} that the transfer function  $\theta_{\widetilde{\Sigma}}$ of $\widetilde{\Sigma}$  belongs to $ \mathbf{S}_{\kappa}(\widetilde{\cU},\widetilde{\cY}),$ which implies now    $\theta_{\Sigma}\in \SK.$ Then the first claim  proved above implies that $(\cX^c)^\perp$ and $ (\cX^o)^\perp$ are Hilbert subspaces.

  If one assumes that $(\cX^c)^\perp$ or $(\cX^o)^\perp$ is a Hilbert space, the identity  \eqref{intse} shows that  $(\cX^s)^\perp$ is a Hilbert space as well. Then the argument above can be applied, and the second claim is proved.

  %

\end{proof}

The \textbf{product} or \textbf{cascade connection} of two  systems
 $\Sigma_1=(A_1,B_1,C_1,D_1;\cX_1,\cU,\cY_1;\kappa_1)$ and $\Sigma_2=(A_2,B_2,C_2,D_2;\cX_2,\cY_1,\cY;\kappa_2)$
is a system $\Sigma_2 \circ \Sigma_1=(T_{\Sigma_2 \circ \Sigma_1};\cX_1 \oplus \cX_2,\cU,\cY;\kappa_1 +\kappa_2)$ such that
 \begin{equation}\label{product1}
  T_{\Sigma_2 \circ \Sigma_1}= \begin{pmatrix}
                             \begin{pmatrix}
                               A_1 & 0 \\
                               B_2C_1 & A_2
                             \end{pmatrix} & \begin{pmatrix}
                                   B_1 \\ B_2D_1
                                 \end{pmatrix} \\
                           \begin{pmatrix}
                                   D_2C_1 & C_2
                                 \end{pmatrix} & D_2D_1
                           \end{pmatrix}: \begin{pmatrix}
                                            \begin{pmatrix}
                                              \cX_1 \\
                                              \cX_2
                                            \end{pmatrix} \\
                                              \cU
                                          \end{pmatrix}\to \begin{pmatrix}
                                            \begin{pmatrix}
                                              \cX_1 \\
                                              \cX_2
                                            \end{pmatrix} \\
                                              \cY
                                          \end{pmatrix}.
 \end{equation} Written in the form \eqref{colli}, one has $\cX= \begin{pmatrix}
                                              \cX_1 \\
                                              \cX_2
                                            \end{pmatrix}$ and
 \begin{equation}\label{prod-oper}  A=\begin{pmatrix}
                               A_1 & 0 \\
                               B_2C_1 & A_2
                             \end{pmatrix},\qquad B=\begin{pmatrix}
                                   B_1 \\ B_2D_1
                                 \end{pmatrix},\qquad C=\begin{pmatrix}
                                   D_2C_1 & C_2
                                 \end{pmatrix},\qquad D=  D_2D_1.       \end{equation}
  Note that $A_2=A\uphar_{\cX_2}$  
  and \begin{equation}\label{product2}
                            \begin{pmatrix}
                               A_1 & 0 & B_1 \\
                               B_2C_1 & A_2 & B_2D_1\\
                                   D_2C_1 & C_2  & D_2D_1
                           \end{pmatrix} = \begin{pmatrix}
                                             I_{\cX_1} & 0 & 0 \\
                                             0 & A_2 & B_2 \\
                                             0 & C_2 & D_2
                                           \end{pmatrix}\begin{pmatrix}
                                             A_1 & 0 & B_1 \\
                                             0 & I_{\cX_2} & 0 \\
                                             C_1 & 0 & D_1
                                           \end{pmatrix}: \begin{pmatrix}
                                              \cX_1 \\
                                              \cX_2
                                            \\
                                              \cU
                                          \end{pmatrix}\to \begin{pmatrix}
                                              \cX_1 \\
                                              \cX_2 \\
                                              \cY
                                          \end{pmatrix}.
                          \end{equation}
                         The product $\Sigma_2 \circ \Sigma_1$ is defined when the incoming space of $\Sigma_2$ is the outgoing space of $\Sigma_1.$ Again, direct computations show that
                          $\theta_{\Sigma_2 \circ \Sigma_1} =\theta_{\Sigma_2}  \theta_{\Sigma_1}$ whenever both functions are defined. For the dual system one has $(\Sigma_2 \circ \Sigma_1)^*=  \Sigma_1^* \circ \Sigma_2^*.$
                          It follows from the identity \eqref{product2}  that the product  $ \Sigma_2 \circ \Sigma_1 $ is
                           conservative (isometric, co-isometric, passive)
                           whenever   $\Sigma_1$ and   $\Sigma_2$ are. Also, if the product is isometric (co-isometric, conservative) and  one factor  of the product is conservative, then the other factor must be isometric (co-isometric, conservative).

                           The product of two systems preserves similarity properties introduced on page \pageref{unsims} in sense that if $\Sigma=\Sigma_2 \circ \Sigma_1$ and $\Sigma'=\Sigma_2' \circ \Sigma_1'$ such that $\Sigma_1$ is unitarily (weakly) similar with $\Sigma_1'$ and $\Sigma_2$ is unitarily (weakly) similar with $\Sigma_2',$ then easy calculations using \eqref{product2} show that $\Sigma$ and $\Sigma'$ are unitarily (weakly) similar.

                            It is known (c.f. e.g. \cite[Theorem 1.2.1]{ADRS}) that if $\Sigma_2 \circ \Sigma_1$ is
                            controllable (observable, simple, minimal), then so are $\Sigma_1$ and $\Sigma_2.$
                             The converse statement is not true. The following lemma  gives necessary and sufficient conditions when the product is observable, controllable or simple. The simple case is handled in \cite[Lemma 7.4]{Saprikinetal}.
 \begin{lemma}\label{product-conds}
  Let $\Sigma_1=(A_1,B_1,C_1,D_1;\cX_1,\cU,\cY_1;\kappa_1), \Sigma_2=(A_2,B_2,C_2,D_2;\cX_2,\cY_1,\cY;\kappa_2)$ and   $\Sigma= \Sigma_2 \circ \Sigma_1.$  Let $\Omega=\overline{\Omega}$ be a symmetric neighbourhood of the origin such that the transfer function $\theta_\Sigma=\theta_{\Sigma_2}\theta_{\Sigma_1}$ of  $\Sigma$ is analytic in  $\Omega.$ Consider the equations
   \begin{align}
     \theta_{\Sigma_2}(z)C_1(I-zA_1)^{-1}x_1&=-C_2(I-zA_2)^{-1}x_2,\quad \text{for all $z \in \Omega$}; \label{pro-ops} \\
   { \theta}_{\Sigma_1}^\#(z)B_2^*(I-zA_2^*) ^{-1}x_2&=-B_1^*(I-zA_1^*)^{-1 } x_1,\quad \text{for all $z \in \Omega$},\label{pro-cont}
   \end{align} where $ x_1 \in \cX_1$ and $x_2 \in \cX_2.$ Then $\Sigma$ is observable if and only if \eqref{pro-ops} has only the trivial solution, and $\Sigma$ is controllable if and only if \eqref{pro-cont} has only the trivial solution. Moreover, $\Sigma$ is simple if and only if  the pair of equations consisting of \eqref{pro-ops} and \eqref{pro-cont} has only the trivial solution. 
 \end{lemma}
\begin{proof}
  Write the system operator $T_{\Sigma_2 \circ \Sigma_1}$ in \eqref{product1} in the form  \eqref{colli}.  
  It follows from \eqref{cont2}--\eqref{simp2} that
  \begin{align}
     x \in (\cX^o)^\perp  &\qekv C(I-zA)^{-1}x=0 \quad \text{for all $z \in \Omega$}; \label{perps2}\\
       x \in (\cX^c)^\perp  &\qekv B^*(I-zA^*)^{-1}x=0 \quad \text{for all $z \in \Omega$};\label{perps1} \\
      x \in (\cX^s)^\perp  &\qekv B^*(I-zA^*)^{-1}x=0 \quad \text{and}\quad  C(I-A)^{-1}x=0 \quad \text{for all $z \in \Omega$}\label{perps3}.
  \end{align} Decompose $x=x_1 \oplus x_2,$ where $x_1 \in \cX_1$ and $x_2 \in \cX_2.$  With respect to the this decomposition, the definition of the main operator $A$ from  \eqref{prod-oper} yields
  $$  (I-zA)^{-1}  = \begin{pmatrix}
                                                   (I_{\cX_1}-zA_1)^{-1} & 0 \\
                                  z(I_{\cX_2}-zA_2)^{-1}B_2C_1(I_{\cX_1}-zA_1)^{-1}   &  (I_{\cX_2}-zA_2)^{-1}
                                               \end{pmatrix}.   $$
  From this relation and \eqref{prod-oper},   it follows that the right hand side of \eqref{perps2} is equivalent to    \begin{align} \begin{pmatrix}
                                   D_2C_1 & C_2
                                 \end{pmatrix} \begin{pmatrix}
                                                   (I_{\cX_1}-zA_1)^{-1} & 0 \\
                                  z(I_{\cX_2}-zA_2)^{-1}B_2C_1(I_{\cX_1}-zA_1)^{-1}   &  (I_{\cX_2}-zA_2)^{-1}
                                               \end{pmatrix}  \begin{pmatrix}
     x_1 \\
     x_2
   \end{pmatrix}=0 \quad \text{for all $z \in \Omega$}
   . \label{perps5}\end{align}
   Similar calculations show that the right hand side of  \eqref{perps1} is equivalent to  \begin{align}
     \begin{pmatrix}
                                   B_1^* & D_1^*B_2^*
                                 \end{pmatrix}\?
                                 \begin{pmatrix}
                                   (I_{\cX_1}-zA_1^*)^{-1} & z(I_{\cX_1}-zA_1^*)^{-1}C_1^*B_2^*(I_{\cX_2}-zA_2^*)^{-1} \\
                                   0 &  (I_{\cX_2}-zA_2^*)^{-1}
                                 \end{pmatrix}
                                \? \begin{pmatrix}
     x_1 \\
     x_2
   \end{pmatrix}=0 \quad \text{for all $z \in \Omega$}
   .\label{perps4}
  \end{align} Expanding the identity \eqref{perps5} and using the definition of the transfer function $$\theta_{\Sigma_2}(z)=D_2 + zC_2(I_{\cX_2}-zA_2)^{-1}B_2$$ one gets that \eqref{perps5} is equivalent to
  \begin{align*}
    \left(  D_2
     +       C_2   z(I_{\cX_2}-zA_2)^{-1}B_2 \right) C_1(I_{\cX_1}-zA_1)^{-1}x_1 &=-  C_2(I_{\cX_2}-zA_2)^{-1}x_2\\ \qekv
     \theta_{\Sigma_2}(z) C_1(I_{\cX_1}-zA_1)^{-1}x_1 &=-  C_2(I_{\cX_2}-zA_2)^{-1}x_2.
  \end{align*} That is, the identity \eqref{perps5} is equivalent to \eqref{pro-ops}. Similar calculations and the identity
  $$  \theta_{\Sigma_1}^\#(z)=D_1^* + zB_1^*(I_{\cX_1}-zA_1^*)^{-1}C_1^*    $$ shows that the identity \eqref{perps4}
 is equivalent to \eqref{pro-cont}. The results follow now by observing that if the system $\Sigma$ is observable, controllable or simple, then, respectively, $(\cX^o)^\perp=\{0\}$,   $(\cX^c)^\perp=\{0\}$ or $(\cX^s)^\perp=\{0\}$.
\end{proof}

 Part {(iii)} of the theorem below with an additional condition that all the realizations are conservative, is proved in \cite[Theorem 7.3, 7.6]{Saprikinetal}. Similar techniques  will be used to expand this result as follows.
 \begin{theorem} \label{preserv}
 Let $\theta \in \SK$ and let $\theta=\theta_r B_r^{-1}=B_l^{-1}\theta_l$ be its Kre\u{\i}n--Langer factorizations.
 Suppose that \begin{align*}\Sigma_{\theta_r}&=(T_{\Sigma_{\theta_r}},\cX^+_r,\cU,\cY,0),& \Sigma_{\theta_l}&=(T_{\Sigma_{\theta_l}},\cX^+_l,\cU,\cY,0), \\ \Sigma_{B_r^{-1}}&=(T_{\Sigma_{B_r^{-1}}},\cX^-_r,\cU,\cU,\kappa),&  \Sigma_{B_l^{-1}}&=(T_{\Sigma_{B_l^{-1}}},\cX^-_l,\cY,\cY,\kappa),\end{align*} are the realizations of $\theta_r,\theta_l,B_r^{-1}$ and $B_l^{-1},$ respectively.  Then:
 \begin{itemize}
   \item[{\rm (i)}]  If $\Sigma_{\theta_r} $ and $ \Sigma_{B_r^{-1}}$ are observable and passive, then so is  $\Sigma_{\theta_r} \circ \Sigma_{B_r^{-1}}$;
   \item[{\rm (ii)}] If $\Sigma_{\theta_l} $ and $ \Sigma_{B_l^{-1}}$ are controllable and passive, then so is  $\Sigma_{B_l^{-1}} \circ   \Sigma_{\theta_l}  $;
   \item[{\rm (iii)}]  If all the realizations described above are simple passive, then so are $\Sigma_{\theta_r} \circ \Sigma_{B_r^{-1}}  $ and $\Sigma_{B_l^{-1}}\circ   \Sigma_{\theta_l} .$
 \end{itemize}
 \end{theorem}
 \begin{proof} 
Suppose first that 
  $ \Sigma_{B_r^{-1}}$ is a  simple passive system and    $ \Sigma_{\theta_r}$ is a passive system.
   The results from \cite[Theorems 9.4 and 10.2]
   {Saprikinetal} show that all the simple passive realizations of $B_r^{-1}$ are conservative and minimal. Thus, the assumptions guarantees that   $ \Sigma_{B_r^{-1}}$ is conservative and minimal. Represent the system operators $T_{\Sigma_{B_r^{-1}}}$ and $T_{\Sigma_{\theta_r}}$ as
\begin{equation}\label{repre}   T_{\Sigma_{B_r^{-1}}}  = \begin{pmatrix}
                                      A_1 & B_1 \\
                                      C_1 & D_1
                                    \end{pmatrix} : \begin{pmatrix}
                                                     \cX^-_r  \\
                                                     \cU
                                                    \end{pmatrix} \to \begin{pmatrix}
                                                     \cX^-_r  \\
                                                     \cU
                                                    \end{pmatrix}  ,\qquad  T_{\Sigma_{\theta_r}}  =\begin{pmatrix}
                                      A_2 & B_2 \\
                                      C_2 & D_2
                                    \end{pmatrix}   \begin{pmatrix}
                                                     \cX^+_r  \\
                                                     \cU
                                                    \end{pmatrix} \to \begin{pmatrix}
                                                     \cX^+_r  \\
                                                     \cY
                                                    \end{pmatrix}.    \end{equation}
  Let $\Omega=\overline{\Omega}$ be a symmetric neighbourhood of the origin  such that $B_r^{-1}$  is analytic in $\Omega.$
   Suppose that $x_1 \in   \cX^-_r$ and $x_2 \in \cX^+_r $ satisfy
  \begin{equation}\label{cannotbe}
       \theta_{r}(z)C_1(I-zA_1)^{-1}x_1=-C_2(I-zA_2)^{-1}x_2,\quad \text{for all $z \in \Omega$}.\end{equation}
       The space $\cX^-_r $  is $\kappa$-dimensional anti-Hilbert space, and all the poles of $B_r^{-1}$ are also  poles of $C_1(I-zA_1)^{-1}x_1.$   Since $ \cX^+_r$ is a Hilbert space, the operator $A_2$ is a Hilbert space contraction,  and $(I-zA_2)^{-1}$ exists for all $z \in \dD.$ That is, the right hand side of \eqref{cannotbe} is holomorphic in $\dD,$ and then so is the left hand side also. Since $\theta_r$ and $B_r$ have no common zeros in the sense of Theorem \ref{krein-langer-fact} and the zeros of $B_r$ are the poles of $B_r^{-1}$, the factor $\theta_r(z)$ cannot cancel out the poles of $C_1(I-zA_1)^{-1}x_1$ (For a more detailed argument, see the proof of \cite[Theorem 7.3]{Saprikinetal}).
       %
       %
       %
        That is,  $ \theta_{r}(z)C_1(I-zA_1)^{-1}x_1    $ is holomorphic in $\dD$ only if   $C_1(I-zA_1)^{-1}x_1\equiv0.$ Then also $C_2(I-zA_2)^{-1}x_2\equiv0,$  and it follows from \eqref{obs2} that  $x_1\in ({\cX^-_r}^o)^\perp$    and $x_2 \in ({\cX^+_r}^o)^\perp.$ Since the system $\Sigma_{B_r^{-1}}$ is minimal, $x_1=0.$ If  the system $\Sigma_{\theta_r}$ is observable, then   $x_2=0,$ and it follows from  Lemma \ref{product-conds} that $\Sigma_{\theta_r} \circ \Sigma_{B_r^{-1}}$ is observable and passive, and part (i) is proven.

        Next suppose that  $x_1$ and $x_2$ satisfy \eqref{cannotbe} and
     \begin{equation}\label{cannotbe1} B^{-1\#}_{r}(z)B_2^*(I-zA_2^*) ^{-1}x_2=-B_1^*(I-zA_1^*)^{-1 } x_1,\quad \text{for all $z \in \Omega$}. \end{equation} The argument above gives  $x_1=0$ and $x_2 \in ({\cX^+_r}^o)^\perp.$ Then,
   $$B^{-1\#}_{r}(z)B_2^*(I-zA_2^*) ^{-1}x_2\equiv0. $$  Since $B^{-1\#}_{r}(z)$ has just the trivial kernel for every $z \in \Omega,$ also $B_2^*(I-zA_2^*) ^{-1}x_2\equiv0. $ 
   The identity  \eqref{cont2} implies now   $x_2 \in ({\cX^+_r}^c)^\perp,$  and therefore $$x_2 \in ({\cX^+_r}^c)^\perp \cap ({\cX^+_r}^o)^\perp =({\cX^+_r}^s)^\perp.  $$ If  the system $\Sigma_{\theta_r}$ is simple, then   $x_2=0,$ and it follows from  Lemma \ref{product-conds} that $\Sigma_{\theta_r} \circ \Sigma_{B_r^{-1}}$ is simple and passive, and the first claim of the part (iii) is proven.
   The other claim  in part (iii) and also part (ii) follow now by considering the dual systems.
 \end{proof}

The  product of the form  $\Sigma_{B_l^{-1}}\circ   \Sigma_{\theta_l} $ does not necessarily preserve observability as is shown in Example \ref{counter}  below. A counter-example is constructed with the help of the following realization result.  For the proof and more details, see \cite[Theorem 2.2.1]{ADRS}.

\begin{lemma}\label{canonical}
Let $S\in \SK$ and let $\cH(S)$ be the Pontryagin space induced by the reproducing kernel \eqref{kernel1}. Then   the system $\Sigma=(A,B,C,D,\cH(S),\cU,\cY ; \kappa)$ where
\begin{equation}\left\{\begin{aligned}\label{canonco}
  &A: h(z) \mapsto \frac{h(z)-h(0)}{z}, \qquad &B&: u \mapsto \frac{S(z)-S(0)}{z}u, \\
  &C:  h(z) \mapsto h(0),\qquad &D&: u \mapsto S(0)u,
\end{aligned} \right. \end{equation} is co-isometric  and observable realization of $S.$ Moreover, $C(I-zA)^{-1}h=h(z)$ for $h\in \cH(S). $
\end{lemma} The system $\Sigma$ in Lemma \ref{canonical} is called a \textbf{canonical co-isometric realization} of $S.$

If the systems $\Sigma_1$ and $\Sigma_2$ in Lemma \ref{product-conds} have additional properties,  a criterion for observability that does not explicitly depend on a system operator can be obtained.

%

  \begin{theorem} \label{obspre}  Let $\Sigma_1\!=\!(A_1,B_1,C_1,D_1;\cX_1,\cU,\cY_1; \? \kappa_1\?)$ and $ \Sigma_2\!=\!(A_2,B_2,C_2,D_2;\cX_2,\cY_1,\cY;\?\kappa_2\?)$ be co-isometric and observable realizations of the functions $S_1 \in \mathbf{S}(\cU,\cY_1)$ and $S_2 \in \mathbf{S}(\cY_1,\cY)$, respectively. Then  $\Sigma= \Sigma_2 \circ \Sigma_1$
  is co-isometric observable realization of $S=S_2S_1$ if and only if the following two conditions hold:
  \begin{itemize}
   \item[{\rm(i)}] $\cH(S)=S_2\cH(  S_1) \oplus \cH(  S_2 );$
   \item[{\rm(ii)}] The mapping
$ h_1 \mapsto S_2h_1$
 is an isometry from $\cH(  S_1) $ to $S_2\cH(  S_1).$
  \end{itemize}\end{theorem}
  \begin{proof}
     Since all co-isometric observable realizations of $S_1$ and $S_2$ are unitarily similar, it can be assumed that $\Sigma_1$ and $\Sigma_2$ are realized as in Lemma \ref{canonical}.  Let $\Omega$ be a neighbourhood of the origin such that $S_1$ and $S_2$ are analytic in $\Omega.$ By combining  Lemma \ref{canonical} and the condition \eqref{pro-ops} in Lemma \ref{product-conds}, it follows that $\Sigma$ is observable if and only if
     \begin{equation}\label{pro-obs2}
        S_2(z)h_1(z)=-h_2(z), \quad h_1 \in \cH(S_1),  \quad h_2 \in \cH(S_2),
     \end{equation}   holds for every $z \in \Omega$ only when $h_1\equiv0$ and $h_2\equiv0.$

     Assume the conditions {\rm(i)} and {\rm(ii)}. 
       Then $S_2(z)h_1(z)=-h_2(z)$ can hold only if $h_2\equiv0.$ Since the mapping $ h_1 \mapsto S_2h_1$ is an isometry, it has only the trivial kernel. Therefore $h_1\equiv0,$  and sufficiency is proven.

       Conversely, assume  that $\Sigma$ is co-isometric and observable. The condition \eqref{pro-obs2} shows that  the  mapping $ h_1 \mapsto S_2h_1$  has only the trivial kernel, and \begin{equation}\label{inters}S_2\cH(  S_1)\cap \cH(  S_2 )=\{0 \}.\end{equation} It now follows from \cite[Theorem 4.1.1]{ADRS}  that $\cH(  S_1) $ and  $S_2\cH(  S_1)$ are  contained contractively in $\cH(  S), $ and  $ h_1 \mapsto S_2h_1$ is a partial isometry. Since it has only the trivial kernel, it is an isometry, and {\rm(ii)} holds. Since \eqref{inters} holds and $\cH(  S_1)$  and $S_2\cH(  S_1)$ are  contained contractively in $\cH(  S), $ a result from
       \cite[Theorem 1.5.3]{ADRS}  shows that   $\cH(  S_1) $ and  $S_2\cH(  S_1)$ are actually  contained isometrically in $\cH(  S).$ Therefore $\cH(  S_1) ^\perp=S_2\cH(  S_1)$ so the condition {\rm(i)} holds and the necessity is proven.  \end{proof}

  The dual version can be obtained by  using the  \textbf{canonical isometric} realizations from \cite[Theorem 2.2.2]{ADRS} or taking adjoint systems in Theorem \ref{obspre}.

\begin{theorem} \label{contpre}
      Let $\Sigma_1\!=\!(A_1,B_1,C_1,D_1;\cX_1,\cU,\cY_1; \? \kappa_1\?)$ and $ \Sigma_2\!=\!(A_2,B_2,C_2,D_2;\cX_2,\cY_1,\cY;\?\kappa_2\?)$ be isometric and controllable realizations of the functions $S_1 \in \mathbf{S}(\cU,\cY_1)$ and $S_2 \in \mathbf{S}(\cY_1,\cY)$, respectively. Then  $\Sigma= \Sigma_2 \circ \Sigma_1$
  is isometric and controllable realization of $S=S_2S_1$ if and only if the following two conditions hold:
  \begin{itemize}
       \item[{\rm(i)}] $\cH({S^\#})={S}^\#_1\cH(  {S}^\#_2) \oplus \cH(  {S}^\#_1 );$
       \item[{\rm(ii)}]  The mapping
$ h_2 \mapsto S^\#_1h_2$
 is an isometry from $\cH(  S^\#_2) $ to $S_1^\#\cH(  S^\#_2).$
  \end{itemize}
\end{theorem}

In the Hilbert state space settings,  a different  criterion   than in Theorems \ref{obspre} and \ref{contpre} was obtained in \cite{Khanh}.
If $\Sigma_1$ and $\Sigma_2$ are simple conservative, a criterion for $\Sigma= \Sigma_2 \circ \Sigma_1$ to be simple conservative was obtained in the Hilbert state space case in \cite{Br1}
and generalized to the Pontryagin state space case in \cite{Saprikinetal}.

Here is the promised counter-example.
\begin{example} \label{counter}
  Let $a \in H^\infty(\dD)$ such that $\| a\|\leq 1$  and let $b(z)=(z-\alpha)/(1-z \bar{\alpha} )$ where $\alpha \in \dD \setminus \{ 0\}.$ Define
   \begin{equation}\label{funct1}
        S(z):=  \frac{1}{\sqrt{2}}  \begin{pmatrix}
                                      a(z) &  \dfrac{1}{b(z)}
                                    \end{pmatrix}, \quad z \in \dD \setminus \{ \alpha\}.
      \end{equation}
      Then  $ S \in \mathbf{S}_{1}(\dC^2,\dC) $ and it has the  left 
       Kre\u{\i}n--Langer factorization 
   \begin{align} S(z)=b^{-1}(z)S_l(z)&=b^{-1}(z)\begin{pmatrix} \frac{1}{\sqrt {2}}a(z)b(z) &  \frac{1}{\sqrt {2}}\end{pmatrix} .
         \end{align}
                      Consider the canonical co-isometric realizations $\Sigma_{b^{-1}}$ and $\Sigma_{S_l}$ of $b^{-1}$ and ${S_l}$, respectively.  
                    It follows    from Theorem \ref{obspre} that if $\Sigma_{b^{-1}} \circ \Sigma_{S_l}   $ is observable, then $\cH(S)=b^{-1}\cH(  S_l) \oplus \cH(  b^{-1}).$ The argument in \cite[p. 149]{ADRS} shows that this is false, so  $\Sigma_{b^{-1}}\circ \Sigma_{S_l}$ is not observable. By considering the adjoint system one obtains a product of type $\Sigma_{S_r} \circ \Sigma_{B^{-1}_r}$ which is not controllable, while  $\Sigma_{S_r} $ and $ \Sigma_{B^{-1}_r}$ are.
\end{example} The function $S$ in  Example \ref{counter} is taken from \cite[p. 149]{ADRS}.

If the realization $\Sigma$ of $\theta =\theta_r B^{-1}_r=B^{-1}_l\theta_l \in \SK$ has additional properties, it can be represented as the product of the form $\Sigma_{\theta_r} \circ \Sigma_{B^{-1}_r}$ or $\Sigma_{B^{-1}_l} \circ  \Sigma_{\theta_l}. $   The following theorem expands the results of \cite[Theorem 7.2
                           ]{Saprikinetal}. 

                           \begin{theorem} \label{reps1}
                             Let $\theta \in \SK$ and $\theta=\theta_rB_r^{-1}=B_l^{-1}\theta_l$ be its Kre\u{\i}n-Langer factorizations.
                              Let $\Sigma_k, k=1,2,3,$ be the realizations of $\theta$ which are respectively conservative, co-isometric and isometric  such that the negative dimension of the state space in each realization is $\kappa$. Then:
                             \begin{itemize}
                             \item[{\rm(i)}]  The realization $\Sigma_1$ can be represented as the products of the form $$\Sigma_1=\Sigma_{\theta_r} \circ \Sigma_{B_r^{-1}}=\Sigma_{B_l^{-1}} \circ \Sigma_{\theta_l},$$ where $\Sigma_{\theta_r}=(T_{\Sigma_{\theta_r}};\cX^{+}_{r},\cU,\cY;0)$ and $\Sigma_{\theta_l}=(T_{\Sigma_{\theta_l}};\cX^{+}_{l},\cU,\cY;0)$ are conservative realizations of the functions $\theta_r$ and $\theta_l$, respectively,  and $ \Sigma_{B_r^{-1}}=(T_{\Sigma_{B_r^{-1}}};\cX^{-}_{r},\cU,\cU;\kappa)$  and $ \Sigma_{B_l^{-1}}=(T_{\Sigma_{B_l^{-1}}};\cX^{-}_{l},\cY,\cY;\kappa)$ are conservative and minimal realizations of the functions $B_r^{-1}$ and  $B_l^{-1},$ respectively.

                               \item[{\rm(ii)}]  The realization $\Sigma_2$ can be represented as the product of the form $$\Sigma_2=\Sigma_{\theta_r} \circ \Sigma_{B_r^{-1}},$$ where $\Sigma_{\theta_r}=(T_{\Sigma_{\theta_r}};\cX^+,\cU,\cY;0)$
                                is  a co-isometric realization of the function $\theta_r$  and $ \Sigma_{B_r^{-1}}=(T_{\Sigma_{B_r^{-1}}};\cX^-,\cU,\cU;\kappa)$
                                is a conservative minimal realization of  $B_r^{-1}.$

                               \item[{\rm(iii)}] The realization $\Sigma_3$ can be represented as the product of the form
                               $$\Sigma_3=\Sigma_{B_l^{-1}} \circ \Sigma_{\theta_l},$$ where $\Sigma_{\theta_l}=(T_{\Sigma_{\theta_l}};\cX^+,\cU,\cY;0)$
                               is an isometric  realization of the function $\theta_l$  and $ \Sigma_{B_l^{-1}}=(T_{\Sigma_{B_l^{-1}}};\cX^-,\cY,\cY;\kappa)$
                               is a conservative minimal realization of  $B_l^{-1}.$
                             \end{itemize}
                           \end{theorem}



                            \begin{proof} The theorem will be proved in two steps. In the first step, it is assumed that $\Sigma_1$ is simple,   $\Sigma_2$ is observable and  $\Sigma_3$ is controllable. In the second step, the general case will be proved by using the results from the first step.

\emph{{Step 1.}}
 \rm{(i)} This is stated essentially in  \cite[Theorem 7.2]{Saprikinetal} but without proof.
  According to \cite[Theorem 4.4]{DLS1},
    $  \Sigma_1=(T_{\Sigma};\cX,\cU,\cY;\kappa)   $ can be represented as the products of the form $$\Sigma_1= \Sigma_{r2}\circ \Sigma_{r1}=\Sigma_{l2}\circ \Sigma_{l1}$$
    such that
   \begin{equation}\label{products} \begin{split}  \Sigma_{r1}&=(T_{\Sigma_{r1}},\cX^{-}_{r},\cU,\cU,\kappa), \qquad \Sigma_{r2}=(T_{\Sigma_{r2}},\cX^{+}_{r},\cU,\cY,0), \\
     \Sigma_{l1}&=(T_{ \Sigma_{l1}},\cX^{+}_{l},\cU,\cY,0), \qquad
    \,\, \Sigma_{l2}=(T_{ \Sigma_{l2}},\cX^{-}_{l},\cY,\cY,\kappa),\end{split}\end{equation} where $\cX^{-}_{r}$ and  $\cX^{-}_{l}$ are $\kappa$-dimensional anti-Hilbert spaces. Subscripts refer "right" and "left", because it will be proved that the factorizations $$\theta =\theta_{\Sigma_{r2}}\theta_{\Sigma_{r1}}=\theta_{\Sigma_{l2}}\theta_{\Sigma_{l1}}$$ of the transfer function $\theta$ of $\Sigma_1$ corresponding to the product representations above are actually Kre\u{\i}n-Langer factorizations.
 Since    all  the realizations in \eqref{products} are simple and conservative, it follows from Lemma \ref{simp-kar} that
  $\theta_{\Sigma_{r2}},\theta_{\Sigma_{l1}}\in \So,
      \theta_{\Sigma_{r1}} \mathbf{S}_{\kappa}(\cU), \theta_{\Sigma_{l2}} \in \mathbf{S}_{\kappa}(\cY),$  and the spaces
       \begin{equation}\label{zerospaces} \cX^{-}_{r} \ominus{\cX^{-}_{r}}^c, \qquad \cX^{-}_{r} \ominus{\cX^{-}_{r}}^o,\qquad \cX^{-}_{l} \ominus{\cX^{-}_{l}}^c,\qquad
           \cX^{-}_{r} \ominus{\cX^{-}_{l}}^o,\end{equation} are Hilbert spaces. But since the state spaces $\cX^{-}_{r}$ and  $\cX^{-}_{l}$ are 
           anti-Hilbert spaces, all the spaces in \eqref{zerospaces} must be the   zero spaces.
            Thus $\Sigma_{r1}$ and $\Sigma_{l2}$ are minimal.
            By using the unitary similarity introduced on page \pageref{unsims} it can be deduced now that
           all co-isometric observable   realizations of $\theta_{r2}$
               and $ \theta_{l1} $ are conservative and minimal,
                and then it follows from \cite[Theorem A3]{ADRS}
             that  $\theta_{r2}$       and $ \theta_{l1} $
                are inverse Blaschke products, which gives the result.

\rm{(ii)} It is known (cf. e.g.  \cite[Theorem 2.4.1]{ADRS})  that  the co-isometric and observable realization
$\Sigma_2=(A,B,C,D;\cX,\cU,\cY;\kappa)$ of the function $\theta$ has a
 simple and conservative dilation
 $\widehat{\Sigma}_{2}=(\widehat{A},\widehat{B},\widehat{C},D;\widehat{\cX},\cU,\cY;\kappa)$ such that
 \begin{equation}\label{condil1}
                                   T_{\widehat{\Sigma}_{2}} 
                                                                            =\begin{pmatrix}
                                                            \begin{pmatrix}
                                                              A_{11} & A_{12} \\
                                                              0 & A
                                                            \end{pmatrix} & \begin{pmatrix}
                                                                              B_1 \\
                                                                              B
                                                                            \end{pmatrix}  \\
                                                            \begin{pmatrix}
                                                              0 & C
                                                            \end{pmatrix} & D  \\

                                                          \end{pmatrix} :  \begin{pmatrix}
                                                          \begin{pmatrix}
                                                                            \cX_0  \\
                                                                            \cX
                                                                          \end{pmatrix} \\ \cU \end{pmatrix} \?\to\?  \begin{pmatrix}
                                                          \begin{pmatrix}
                                                                            \cX_0  \\
                                                                            \cX
                                                                          \end{pmatrix} \\ \cY \end{pmatrix},
                                 \end{equation}
 where $\cX_0$ is a Hilbert space.
  By \cite[Theorem 7.7]{Saprikinetal}, there exist
   unique fundamental decompositions $\cX=\cX^+ \oplus \cX^-$ and
                                 $\widehat{\cX}=\widehat{\cX}^+\oplus\widehat{ \cX}^{-}$
   such that $A\cX^+\subset\cX^+$ and $\widehat{A} \widehat{\cX}^+\subset \widehat{\cX}^+.$
   Then $(\cX_0  \oplus\cX^+ ) \oplus \cX^- $ is a fundamental decomposition
   of $\widehat{\cX},$ and for $x_0 \in \cX_0  $ and $x_+ \in \cX^+ $
                                 \begin{equation}\label{condil2}
                                   \widehat{A}(x_0 \oplus x_+) =  \begin{pmatrix}
                                                              A_{11} & A_{12} \\
                                                              0 & A
                                                            \end{pmatrix}\begin{pmatrix}
                                                                              x_0 \\
                                                                              x_+
                                                                            \end{pmatrix} =\begin{pmatrix}
                                                                              A_{11}x_0 + A_{12}x_+  \\
                                                                              Ax_+
                                                                            \end{pmatrix} \in \begin{pmatrix}
                                                                              \cX_0 \\
                                                                              \cX^+
                                                                            \end{pmatrix}.
                                 \end{equation}
  This yields $\widehat{\cX}^+ = \cX_0 \oplus
  \cX^+$ and $\widehat{\cX}^{-}={\cX}_{2}^{-}$.
    Part {\rm(i)} shows that  $\widehat{\Sigma}_2$ can be represented as
     $\widehat{\Sigma}_2= \widehat{\Sigma}_{\theta_r} \circ \widehat{\Sigma}_{B_r^{-1}}. $ The transfer functions of the components are $\theta_r$ and  $B_r^{-1}$, respectively, and $\widehat{\Sigma}_{\theta_r} $ is simple and conservative and $\widehat{\Sigma}_{B_r^{-1}}$ is conservative  and  minimal.
    It follows from \cite[Theorem 7.7]{Saprikinetal}  that the  state spaces of  $\widehat{\Sigma}_{\theta_r}$
           and $\widehat{\Sigma}_{B_r^{-1}}$ are $\widehat{\cX}^+$ and $\cX^-$, respectively. Thus
     $$ \widehat{\Sigma}_{B_r^{-1}}=(A_1,B_1,C_1,D_1;\cX^-,\cU,\cU;\kappa), \qquad
      \widehat{\Sigma}_{\theta_r}=(A_2,B_2,C_2,D_2;\widehat{\cX}^+,\cU,\cY;0).$$
         Now the representation $ \widehat{\Sigma}_{\theta_r} \circ \widehat{\Sigma}_{B_r^{-1}}$,  equation \eqref{product2} and the representation \eqref{condil1} yield
\begin{align*}\label{dilations1}
                &T_{\widehat{\Sigma}_2}= \begin{pmatrix}
                               I_{ \cX^-} & 0 & 0 \\
                               0 & A_2 & B_2\\
                                   0 & C_2  & D_2
                           \end{pmatrix} \begin{pmatrix}
                               A_1 & 0 & B_1 \\
                                0& I_{\widehat{ \cX}^+} &0\\
                                   C_1 & 0  & D_1
                           \end{pmatrix}
                           \!:\! \begin{pmatrix}
                \cX^-    \\   \widehat{ \cX}^+ \\
                    \cU
                  \end{pmatrix}\!\to\!  \begin{pmatrix}
                  { \cX}^-   \\  \widehat{ \cX}^+\\
                    \cY
                  \end{pmatrix}                                  \qekv  \\
              &T_{\widehat{\Sigma}_2}=    \begin{pmatrix}
                                      I_{\cX^-} & 0 & 0 & 0 \\
                                    0&    P_{\cX_0} A_2\uphar_{\cX_0}   & P_{\cX_0} A_2\uphar_{\cX^+}&    P_{\cX_0}B_2 \\
                                    0  &     P_{\cX^+}  A_2  \uphar_{\cX_0}   & P_{\cX^+}  A_2 \uphar_{\cX^+} & P_{\cX^+}B_2\\
                                      0  &    0   & C_2 & D_2
                                    \end{pmatrix}  \begin{pmatrix}
                                                     A_1 & 0 & 0& B_1  \\
                                                     0 & I_{\cX^0} & 0 & 0 \\
                                                     0 & 0 & I_{\cX^+} & 0 \\
                                                     C_1 & 0 & 0 & D_1
                                                   \end{pmatrix}
                                      \!:\! \begin{pmatrix}  { \cX}_2^-  \\
               { \cX}_0     \\{ \cX}^+ \\
                    \cU
                  \end{pmatrix} \!\to\! \begin{pmatrix}  { \cX}_2^-  \\
               { \cX}_0     \\{ \cX}^+ \\
                    \cY
                  \end{pmatrix}.
\end{align*}
By using the representation above and \eqref{dilation}--\eqref{dil-res}, it follows  that $$\res_{\widehat{\cX} \to {\cX}}\widehat{\Sigma}_2 =\res_{\widehat{\cX} \to {\cX}} ( \widehat{\Sigma}_{\theta_r} \circ \widehat{\Sigma}_{B_r^{-1}} ) =
                                                                               \left( \res_{\widehat{\cX}^{+} \to \cX^{+}}\widehat{\Sigma}_{\theta_r} \right) \circ    
                                                                               \widehat{\Sigma}_
                                                                               {B_r^{-1}}
                                                                               =\Sigma_2.$$ Define $\widehat{\Sigma}_
 {B_r^{-1}}:={\Sigma}_
 {B_r^{-1}}$ and  $ \res_{\widehat{\cX}^{+} \to \cX^{+}}\widehat{\Sigma}_{\theta_r}:={\Sigma}_{\theta_r}.$ 
 Since $\Sigma_2$ is co-isometric and observable and
  ${\Sigma}_
 {B_r^{-1}}$ is minimal and conservative, ${\Sigma}_{\theta_r} $ must be co-isometric and observable. That is,  $\Sigma_2={\Sigma}_{\theta_r}  \circ    
 {\Sigma}_
    {B_r^{-1}}$ is the desired representation.

     {\rm (iii) } This can be done by using \cite[Theorem 2.4.3]{ADRS} and then  proceeding along the lines of the proof of  {\rm (ii) }.

  \emph{Step 2.}
                           {\rm(i)} Denote $\Sigma_1=(A,B,C,D;\cX;\cU,\cY;\kappa).$ Since the index of the transfer function $\theta$ coincides with the negative index of $\cX,$ Lemma \ref{simp-kar} shows that $ (\cX^s)^\perp$ is a Hilbert space. It easily follows from \eqref{simp1} that $C (\cX^s)^\perp =\{0\},B^{*} (\cX^s)^\perp =\{0\}, A\cX^s\subset \cX^s $ and $A(\cX^s)^\perp \subset(\cX^s)^\perp. $ This implies that  the system operator has the representation
 \begin{equation}\label{projection1}
   T_{\Sigma_1} =\begin{pmatrix}
                 \begin{pmatrix}
                   A_1 & 0 \\
                   0 & A_0
                 \end{pmatrix} & \begin{pmatrix}
                                   0 \\
                                   B_0
                                 \end{pmatrix}\\
                                  \begin{pmatrix}
                                   0 &
                                   C_0
                                 \end{pmatrix} & D
                \end{pmatrix}: \begin{pmatrix}
                    \begin{pmatrix}(\cX^s)^\perp \\ { \cX}^s \end{pmatrix}\\
                    \cU
                  \end{pmatrix}  \to \begin{pmatrix}
                  \begin{pmatrix}(\cX^s)^\perp \\ { \cX}^s \end{pmatrix}\\
                    \cY
                  \end{pmatrix}.
 \end{equation} Easy calculations show that the restriction $$\res_{\cX\to \cX^s } \Sigma_1 =(A_0,B_0,C_0,D; \cX^s,\cU,\cY;\kappa):=\Sigma_0$$ is conservative and simple. Step 1 (i)   shows that
 $\Sigma_0 = \Sigma_{\theta_r} \circ \Sigma_{B_r^{-1}}=\Sigma_{B_l^{-1}}\circ \Sigma_{\theta_l}, $ where
   \begin{align*}
   \Sigma_{\theta_r}    &=(T_{\Sigma_{\theta_r}};\cX_r^{s+},\cU,\cY;0), \qquad  &\Sigma_{B_r^{-1}} &=(T_{\Sigma_{B_r^{-1}}};\cX_r^{s-},\cU,\cU;\kappa),
   \\ \Sigma_{\theta_l} &=(T_{\Sigma_{\theta_l}};\cX_l^{s+},\cU,\cY;0),  \qquad&\Sigma_{B_l^{-1}}  &=(T_{\Sigma_{B_l^{-1}}};\cX_l^{s-},\cY,\cY;\kappa). \end{align*}
    The spaces $\cX_r^{s-}$ and $\cX_l^{s-}$ are $\kappa$-dimensional anti-Hilbert spaces, $\Sigma_{\theta_r}$    and $\Sigma_{\theta_l}$ are conservative and  simple and $\Sigma_{B_r^{-1}}$ and $\Sigma_{B_l^{-1}}$ are conservative and minimal.
  It can be now deduced that $\cX$ has the fundamental decompositions $( (\cX^s)^\perp\oplus\cX_r^{s+}   ) \oplus \cX_r^{s-}$
 and $( (\cX^s)^\perp \oplus \cX_l^{s+})  \oplus \cX_l^{s-}.$  Moreover,
 \begin{equation*}
 A((\cX_r^s)^\perp \oplus \cX_r^{s+}) \subset (\cX_r^s)^\perp \oplus \cX_r^{s+}, \quad A \cX_l^{s-} \subset\cX_l^{s-}.
 \end{equation*} Similar calculations as in the proof of Step 1 \rm{(ii)} show that
 \begin{equation*}
  \dil \;
   \Sigma_0 =\left(\dil \;
   \Sigma_{\theta_r}\right) \circ \Sigma_{B_r^{-1}} =  \Sigma_{B_l^{-1}} \circ  \left(\dil \;
    \Sigma_{\theta_l} \right) =\Sigma_1.  \end{equation*} Since $\Sigma_1,$ $\Sigma_{B_l^{-1}}$ and $\Sigma_{B_r^{-1}}$  are conservative, $\dil \;
   \Sigma_{\theta_r} $ and $\dil \;
    \Sigma_{\theta_l}$ must be conservative. Moreover, the state spaces $(\cX^s)^\perp\oplus\cX_r^{s+}   $ and $  (\cX^s)^\perp \oplus \cX_l^{s+}$ of $\dil \;
   \Sigma_{\theta_r} $ and $\dil \;
    \Sigma_{\theta_l},$ respectively, are Hilbert spaces. That is, $\Sigma_1=\left(\dil \;
   \Sigma_{\theta_r}\right) \circ \Sigma_{B_r^{-1}} $ and $\Sigma_1= \Sigma_{B_l^{-1}} \circ  \left(\dil \;
    \Sigma_{\theta_l} \right) $  are the desired representations.

    {\rm(ii)} Denote $\Sigma_2=(A,B,C,D;\cX,\cU,\cY;\kappa).$ Lemma \ref{simp-kar} show  that $(\cX^o)^\perp$ is a Hilbert space. From the identity \eqref{obs1} it follows easily that $
      A(\cX^o)^\perp \subset (\cX^o)^\perp $ and $C(\cX^o)^\perp =\{0\}.$ This implies that  the system operator  can be represented as
    \begin{equation}\label{projection2}
   T_{\Sigma_2} =\begin{pmatrix}
                   A_1 & A_2 & B_1 \\
                  0 & A_0 & B_0 \\
                  0 & C_0 & D
                \end{pmatrix}: \begin{pmatrix}
                    (\cX_2^o)^\perp \\ { \cX}_2^o\\
                    \cU
                  \end{pmatrix}  \to \begin{pmatrix}
                    (\cX_2^o)^\perp \\ { \cX}_2^o\\
                    \cY
                  \end{pmatrix}.
 \end{equation}  Moreover,  the restriction $$\res_{\cX_2 \to \cX_2^o } \Sigma_2 =(A_0,B_0,C_0,D; \cX_1^o,\cU,\cY;\kappa):=\Sigma_0$$ is co-isometric and observable. Step 1 {\rm(ii)} shows that $\Sigma_0$ has the representation $\Sigma_0 = \Sigma_{\theta_r} \circ \Sigma_{B_r^{-1}}$ such that the components $$ \Sigma_{\theta_r} =(T_{\Sigma_{\theta_r}},\cX^{o+},\cU,\cY,0)   , \quad \Sigma_{B_r^{-1}} =(T_{\Sigma_{B_r^{-1}}},\cX^{o-},\cU,\cU,\kappa)$$  have the properties introduced in Part 1 {\rm(ii)}. The final statement is obtained by proceeding as in the proof of {\rm(i)}.

  {\rm(iii)} The proof is similar to the proofs of {\rm(i)} and  {\rm(ii)} and hence the details are omitted.  \end{proof}

\begin{proposition} \label{invariants}
  Suppose that $A \in \cL(\cU,\cY)$ is the main operator of  a passive system $\Sigma=(T_\Sigma;\cX,\cU,\cY;\kappa)$ such that the index of the transfer function of $\Sigma$ is $\kappa.$ Then there exist unique fundamental decompositions
 $\cX= \cX_1^+ \oplus \cX_1^-= \cX_2^{+} \oplus \cX_2^{-}$ such that $A\cX_1^+ \subset \cX_1^+$ and $A\cX_2^{-}\subset \cX_2^{-},$ respectively.
\end{proposition}
\begin{proof}
  Embed the system $\Sigma$ in a conservative system $\widetilde{\Sigma}=(T_{\widetilde{\Sigma}},\cX,\widetilde{\cU},\widetilde{\cY},\kappa)$ without changing the  main operator and  the state space. Now the first identity in \eqref{transemb} shows that the transfer function $\theta_{\widetilde{\Sigma}}$ of $\widetilde{\Sigma}$ has the same amount of poles (counting multiplicities) as  the transfer function of the original system. Hence it follows from 
  Lemma \ref{kar-gen} that the index of
   $\theta_{\widetilde{\Sigma}}$ is $\kappa.$ The representations in  Theorem \ref{reps1} {\rm(i)} combined with the decomposition of the main operator $A$ in \eqref{transemb} give  the  claimed fundamental decompositions. 
    The decomposition $\cX_1^+ \oplus \cX_1^- $ corresponds to the one induced by the product representation $\widetilde{\Sigma}=\Sigma_{\widetilde{\theta}_r} \circ \Sigma_{\widetilde{B}_r^{-1}},$ where $\widetilde{\theta} =\widetilde{\theta}_r \widetilde{B}_r^{-1}  $ is the right Kre\u{\i}n-Langer factorization of $\widetilde{\theta}.$ Similarly,  the decomposition $\cX_2^+ \oplus \cX_2^- $ corresponds to the one induced by the product representation $\widetilde{\Sigma}= \Sigma_{\widetilde{B}_l^{-1}}    \circ   \Sigma_{\widetilde{\theta}_l},$ where $\widetilde{\theta} =\widetilde{B}_l^{-1}\widetilde{\theta}_l   $ is the left Kre\u{\i}n-Langer factorization of $\widetilde{\theta}.$

   To prove the uniqueness, the fact that $A$ has  no negative eigenvector with corresponding eigenvalue modulus one is needed. To this end, assume that $Ax=\lambda x$ for some $x \in \cX$ and $\lambda \in \dT.$ Consider again a conservative embedding $\widetilde{\Sigma}$ of $\Sigma,$ and represent $\widetilde{\Sigma}$ as in \eqref{embed} . Then,
   $$\begin{pmatrix}
       A&\widetilde{B} \\
       \widetilde{C} & \widetilde{D}
     \end{pmatrix} \begin{pmatrix}
       x \\
       0
     \end{pmatrix} =\begin{pmatrix}
       \lambda x \\
       \widetilde{C}x
     \end{pmatrix}.  $$ Since $\widetilde{\Sigma}$ is conservative, the system operator $T_{\widetilde{\Sigma}}$ of $\widetilde{\Sigma}$ is unitary. Therefore  $\langle x, x\rangle_{\cX}=  \langle \lambda x, \lambda x\rangle_{\cX} +  \langle \widetilde{C} x, \widetilde{C} x\rangle_{\widetilde{\cY}},$ and since $\widetilde{\cY}$ is a Hilbert space and $|\lambda|=1,$ it must be $\widetilde{C}x=0.$ Then, $\widetilde{C}A^nx=\lambda^n\widetilde{C}x=0$ for any $n \in \dN_0.$ That is, $x\in ( \widetilde{\cX}^o)^\perp,$ where $ \widetilde{\cX}^o$ is the observable subspace of the system $\widetilde{\Sigma}.$ Since the index of $\widetilde{\theta}$ is $\kappa, $ the subspace $ ( \widetilde{\cX}^o)^\perp$ is a Hilbert space by Lemma \ref{simp-kar}, and $x$ must be non-negative.

     Suppose now that $\cX^+ \oplus \cX^-$ is some other fundamental decomposition of $\cX$ such that $A \cX^- \subset  \cX^-.$ It will be shown that  $\cX^- \subset \cX_2^-,$ since then $\cX^- = \cX_2^-$ because these subspaces have the same finite dimension, and thus   $\cX^+ \oplus \cX^-$ is equal to $\cX_2^+ \oplus \cX_2^-.$ It suffices to show that $\cX_2^-$ contains all  generalized eigenvectors of $A\uphar_{\cX^-}.$ Let $x$ be a non-zero vector in  $\cX^-$ such that $(A-\lambda I )^nx=0$ for some  $\lambda \in \dC$ and $ n \in \dN.$ Since $\cX_2^- $ is an anti-Hilbert space and $A\uphar_{\cX^-}$ is a contraction, $|\lambda| \geq1.$   The fact proved above gives now $|\lambda| > 1.$  Represent the vector $x$ in the form $x=x^+ +x^-,$ where $  x^{\pm} \in \cX_2^{\pm}.$ Since $A\cX_2^-\subset \cX_2^-,$ the operator $A$ has a block representation
     $$A =\begin{pmatrix}
            A_{11} & 0 \\
            A_{12} & A_{22}
          \end{pmatrix} : \begin{pmatrix}
                            \cX_2^+ \\
                            \cX_2^-
                          \end{pmatrix} \to \begin{pmatrix}
                            \cX_2^+ \\
                            \cX_2^-
                          \end{pmatrix} .       $$
                          Since $A^*$ is also a contraction, $A_{11}^*$ is a Hilbert space contraction, and therefore $A_{11}$ must be a contraction.
                          Now
                          $$(A-\lambda I )^nx =\begin{pmatrix}
           (A_{11} -\lambda I_{\cX_2^+} )^n & 0 \\
           f(n) & ( A_{22}-\lambda I_{\cX_2^-})^n
          \end{pmatrix}\begin{pmatrix}
                            x^+ \\
                            x^-
                          \end{pmatrix} =\begin{pmatrix}
                          0 \\
                            0
                          \end{pmatrix}, $$ where $f(n)$ is an operator depending on $n.$ This implies $  (A_{11} -\lambda  I_{\cX_2^+} )^n x^+=0,$ but since $A_{11}$ is a Hilbert space contraction and $|\lambda| >1,$ it must be $x^+=0.$ Hence $x=x^-\in \cX_2^-,$ an the uniqueness of the decomposition $\cX= \cX_2^+ \oplus  \cX_2^-$  is proved. The uniqueness of   the decomposition $\cX= \cX_1^+ \oplus  \cX_1^-$ can be proved by using the fact $A^* \cX_1^- \subset  \cX_1^-, $ and then proceeding as above.
   \end{proof}
Proposition \ref{invariants}  is a generalization of \cite[Theorem 7.7]{Saprikinetal} in a sense that the condition that the system is simple can be  relaxed. As proved, it suffices that the orthocomplement $(\cX^s)^\perp$ of the simple subspace is a Hilbert space, see Lemma \ref{simp-kar}.
The proof of Proposition \ref{invariants} follows the lines of the proof of  \cite[Theorem 7.7]{Saprikinetal}.

 The results of  Theorem \ref{reps1} (i) cannot be extended to isometric or co-isometric systems as the next example shows.
  \begin{example}\label{norep}
    Let $S$ be as in Example \ref{counter} and let $\Sigma$ be any co-isometric observable realization of $S.$
    Suppose  that   $\Sigma=\Sigma'_{b^{-1}} \circ \Sigma'_{S_l}$ for some  co-isometric observable realizations of $b^{-1}$ and $S_l.$ 
    Then the realizations $\Sigma'_{b^{-1}}$ and $\Sigma'_{S_l}$ are unitarily similar, respectively, with the  canonical co-isometric observable realizations $\Sigma_{b^{-1}}$ of $b^{-1}$ and $ \Sigma_{S_l}$ of $S_l$. 
     An easy calculation shows that $\Sigma'_{b^{-1}} \circ \Sigma'_{S_l}$ is unitarily similar with
    $\Sigma_{b^{-1}}\circ \Sigma_{S_l},$ which is a contradiction since  $\Sigma_{b^{-1}}\circ \Sigma_{S_l}$ is not observable by Example \ref{counter}. Thus $\Sigma$ cannot be represented as a product of the form $\Sigma'_{b^{-1}} \circ \Sigma'_{S_l}.$
   \end{example}

\section{Stable systems and zero defect functions}\label{sec-defect}
A contraction $A \in \cL(\cX),$ where $\cX$ is a Hilbert space, belongs to the classes $C_{0\,\cdot}$ or $C_{\cdot\,0}$   if, respectively, $\lim_{n \to \infty} A^nx=0$ or $ \lim_{n \to \infty} {A^*}^nx=0 $  for every $x \in \cX.$ The class $C_{00}$ is defined to be  $ C_{0\,\cdot}  \cap C_{\cdot\,0}.$  A system with a Hilbert state space   is said to be \textbf{strongly stable} (\textbf{strongly co-stable}, \textbf{strongly bi-stable}) if the main operator of the system belongs to  $C_{0\,\cdot}$  ($C_{\cdot\,0}$ , $C_{00} $). When the state space $\cX$ is a Pontryagin space, stability cannot be defined verbatim, because for any contractive $A \in \cL(\cX), $ the equality  $\lim_{n \to \infty} A^nx=0 $ does not hold for any negative vector $x.$ The stability property can therefore hold only in certain Hilbert subspaces. The following definition of stability generalizes and expands \cite[Definition 9.1]{Saprikinetal}.

\begin{definition}
  Let $\Sigma=(T_{\Sigma};\cX,\cU,\cY;\kappa)$ be a passive system with the main operator $A$  such that $\theta_{\Sigma}\in \SK.$ Let $\cX= \cX_1^+ \oplus \cX_1^-= \cX_2^{+} \oplus \cX_2^{-}$ be the unique fundamental decompositions of $\cX$ introduced in Proposition \ref{invariants} such that $A\cX_1^+ \subset \cX_1^+$ and $A\cX_2^{-}\subset \cX_2^{-}.$ Then:
  \begin{itemize}
    \item[{\rm (i)}] $\Sigma$ belongs to class $\mathbf{P}^{\kappa}_{0\,\cdot}$ if $A\uphar_{\cX_1^+} \in  C_{0\,\cdot}$;
    \item[{\rm (ii)}] $\Sigma$ belongs to class $\mathbf{P}^{\kappa}_{\cdot\,0}$ if $A^*\uphar_{ \cX_2^{+}} \in  C_{0\,\cdot}$;
    \item[{\rm (iii)}] $\Sigma$ belongs to class $\mathbf{P}^{\kappa}_{00}$ if $A\uphar_{\cX_1^+} \in  C_{00}$;
    \item[{\rm (iv)}]    $\Sigma$ belongs to class $\mathbf{C}^{\kappa}_{0\,\cdot}$ if $\Sigma$ is  simple conservative and $\Sigma \in \mathbf{P}^{\kappa}_{0\,\cdot}$;
    \item[{\rm (v)}]   $ \Sigma$ belongs to class $\mathbf{C}^{\kappa}_{\cdot\,0}$ if $\Sigma$ is  simple conservative and $\Sigma \in \mathbf{P}^{\kappa}_{\cdot\,0}$;
    \item[{\rm (vi)}]  $\Sigma$ belongs to class $\mathbf{C}^{\kappa}_{00}$ if $\Sigma$ is  simple conservative and $\Sigma \in \mathbf{P}^{\kappa}_{00}$;
       \item[{\rm (vi)}]  $\Sigma$ belongs to class $\mathbf{I}^{\kappa}_{0\,\cdot}$ if $\Sigma$ is  controllable isometric and $\Sigma \in \mathbf{P}^{\kappa}_{0\,\cdot}$;
    \item[{\rm (vii)}]    $\Sigma$ belongs to class ${\mathbf{I}^{*}}^{\kappa}_{\cdot\,0}$ if $\Sigma$ is  observable co-isometric and $\Sigma \in \mathbf{P}^{\kappa}_{\cdot\,0}$;
  \end{itemize}
\end{definition} The classes $\mathbf{P}^{\kappa}_{00}$ and $\mathbf{C}^{\kappa}_{00}$ are defined in \cite[Definition 9.1]{Saprikinetal}, as well as the class $\mathbf{P}^{\kappa}_{00}$ with the additional condition that $\Sigma$ must be simple. It will be shown later that the realizations in the classes  $\mathbf{C}^{\kappa}_{00}$,  $\mathbf{I}^{\kappa}_{0\,\cdot}$ and ${\mathbf{I}^{*}}^{\kappa}_{\cdot\,0}$ are minimal, the realizations in $\mathbf{C}^{\kappa}_{0\,\cdot}$ are observable and the realizations in $\mathbf{C}^{\kappa}_{\cdot\,0}$ are controllable.
\begin{theorem}\label{consinners}
  A simple conservative system $\Sigma=(A,B,C,D;\cX,\cU,\cY;\kappa)$ belongs to
  \begin{itemize}
    \item[{\rm(i)}]  $\mathbf{C}^{\kappa}_{0\,\cdot}$ if and only if $\theta_{\Sigma}$ has isometric boundary values a.e.;
    \item[{\rm(ii)}]  $\mathbf{C}^{\kappa}_{\cdot\,0}$ if and only if $\theta_{\Sigma}$ has co-isometric boundary values a.e.;
    \item[{\rm(iii)}]   $\mathbf{C}^{\kappa}_{00}$ if and only if $\theta_{\Sigma}$ has unitary boundary values a.e.
  \end{itemize}
\end{theorem}  In the Hilbert state space case, i.e. $\kappa=0,$ the result is known and goes back essentially to \cite{SF}.     For $\kappa>0,$ part {\rm(iii)} is first proved in \cite[Theorem 9.2]{Saprikinetal}.
\begin{proof}
    Since the results hold for $\kappa=0, $ it suffices to prove the them in case $\kappa>0.$  Consider the representations $\Sigma = \Sigma_{\theta_r} \circ \Sigma_{B^{-1}_r} =  \Sigma_{B^{-1}_l}\circ \Sigma_{\theta_l}  $  as in Theorem \ref{reps1}. Now the results  follow by observing that the main operator of $\Sigma_{\theta_r}$ is $A\uphar_{\cX_1^+}$ and the main operator of  $\Sigma_{\theta_l}^*$ is $A^*\uphar_{ \cX_2^{+}},$ and then using the case $\kappa=0.$
\end{proof}

In Section \ref{sec-pont}, the notions of defect functions were introduced.
If the right or the left defect function  of $\theta \in \SK$ is identically equal to zero,  the realizations of $\theta$ have some strong structural properties.
\begin{lemma} \label{conservs}
   For a simple conservative system $\Sigma=(A,B,C,D;\cX,\cU,\cY;\kappa)$ with the transfer function $\theta \in \SK$, the following statements hold:
   \begin{itemize}
    \item[{\rm(i)}]  $\Sigma$ is controllable if and only if $\psi_{\theta}\equiv0$;
     \item[{\rm(ii)}]  $\Sigma$ is observable if and only if $\varphi_{\theta}\equiv0$;
    \item[{\rm(iii)}]  $\Sigma$ is minimal if and only if $\psi_{\theta}\equiv0$ and $\varphi_{\theta}\equiv0$.
   \end{itemize}
\end{lemma}
\begin{proof}
   For the case $\kappa=0,$ see \cite[Corollary 6.4]{AHS2} or  \cite[Theorem 1]{ArNu2}.  For $\kappa >0,$ consider the representations $\Sigma = \Sigma_{\theta_r} \circ \Sigma_{B^{-1}_r} =  \Sigma_{B^{-1}_l}\circ \Sigma_{\theta_l}  $  as in Theorem \ref{reps1}. If $\Sigma$ is controllable, then so is $ \Sigma_{\theta_l}$ and from case $\kappa=0$ it follows that  $\psi_{\theta_l}\equiv0.$ Now the identity \eqref{rightdefect} implies that $\psi_{\theta}\equiv 0.$
   Conversely, if $\psi_{\theta}\equiv 0,$ the identity \eqref{rightdefect} shows that also $\psi_{\theta_l}\equiv0,$ and  from the case $\kappa=0$ it follows that $\Sigma_{\theta_l}$ is controllable. By Theorem \ref{reps1} (i), $ \Sigma_{B^{-1}_l}$ is minimal. Then it follows from Theorem \ref{preserv} that $\Sigma =  \Sigma_{B^{-1}_l}\circ \Sigma_{\theta_l}  $ is controllable, and part {\rm(i)} is proven. Proof of  part {\rm(ii)} is similar, and  part {\rm(iii)} follows by combining {\rm(i)} and  {\rm(ii)}.
\end{proof}

 The following theorem in the Hilbert state space case was obtained in
\cite[Theorem 1.1]{AHS2}. The proof therein was based on the block parametrization of the system operator. The proof given here for the general case  is based on the existence of minimal passive realizations. It also uses some techniques appearing in the proof of \cite[Theorem 1]{Arov} and, in addition,   implements  the product representations provided in Theorem \ref{reps1}.
\begin{theorem} \label{sepontulos}
  Let $\Sigma=(A,B,C,D,\cX,\cU,\cY,\kappa)$ be a passive system with the transfer function  $\theta.$  
   Then:
  \begin{itemize}
    \item[{\rm (i)}] If $\Sigma$ is controllable and $\varphi_{\theta}\equiv0,$ then $\Sigma$ is isometric and minimal.
     Moreover, if $\theta$ has isometric boundary values a.e., then  $\Sigma \in \mathbf{I}^{\kappa}_{0\,\cdot}.$
    \item[{\rm (ii)}] If $\Sigma$ is observable and $\psi_{\theta}\equiv0,$ then $\Sigma$ is co-isometric and minimal.
     Moreover, if $\theta$ has co-isometric boundary values a.e., then  $\Sigma \in {\mathbf{I}^*}^{\kappa}_{\cdot\,0}.$
    \item[{\rm (iii)}] If $\Sigma$ is simple and $\varphi_{\theta}\equiv0$ and $\psi_{\theta}\equiv0,$ then $\Sigma$ is conservative and minimal.
     Moreover, if $\theta$ has unitary boundary values a.e., then  $\Sigma \in  \mathbf{C}^{\kappa}_{00}.$
  \end{itemize}
\end{theorem}

\begin{proof} (i)
  Denote the system operator of $\Sigma$ by  $T,$ and consider the Julia embedding $\widetilde{\Sigma}$ of the system $\Sigma.$ This means that the corresponding system operator is  a unitary operator of the form
   \begin{equation} T_{\widetilde{\Sigma}} = \begin{pmatrix}
    A & \begin{pmatrix} B & D_{T_{,1}^*} \end{pmatrix} \\
                               \begin{pmatrix}  C \\   D_{T_{,1}}^*  \end{pmatrix} &  \begin{pmatrix}
                                      D &  D_{T_{,2}^*} \\
                                       D_{T_{,2}}^* & -L^*
                                    \end{pmatrix}
                             \end{pmatrix}  : \begin{pmatrix}
                                            \cX \\
                                          \begin{pmatrix}  \cU \\
                                            \sD_{T^*}\end{pmatrix}
                                          \end{pmatrix} \to \begin{pmatrix}
                                            \cX \\\begin{pmatrix}
                                            \cY \\
                                            \sD_{T}\end{pmatrix}
                                          \end{pmatrix},
  \end{equation} where  $$ D_{T^*} = \begin{pmatrix}
                                      D_{T_{,1}^*}\\
                                      D_{T_{,2}^*}
                                    \end{pmatrix},\quad D_{T} = \begin{pmatrix}
                                      D_{T_{,1}}\\
                                      D_{T_{,2}}
                                    \end{pmatrix},\quad   D_{T^*}D_{T^*}^*=I_{\cX}-TT^*,\quad   D_{T}D_{T}^*=I_{\cX}-T^*T,  $$  such that $D_{T}$ and $D_{T^*}$ have zero kernels. The transfer function of the embedded system is
  \begin{equation*}\label{embtrans} {\theta_{\widetilde{\Sigma}}}(z)=
   \begin{pmatrix}
                                                       D + zC(I-zA)^{-1}B & D_{T_{,2}^*}+zC(I-zA)^{-1} D_{T^*_{,1}} \\
                                                          D_{T_{,2}}^* +zD_{T_{,1}}^*(I-zA)^{-1}B &-L^* + zD_{T_{,1}}^*(I-zA)^{-1} D_{T^*_{,1}}
                                                     \end{pmatrix}= \begin{pmatrix}
                                                                                           \theta (z) & \theta_{12}(z)\\
                                                                                            \theta_{21}(z) &  \theta_{22}(z)
                                                                                         \end{pmatrix}.
  \end{equation*} Notice that $\theta,\theta_{12}, \theta_{21}$ and $\theta_{22}$ all are generalized Schur functions. Because
$
    I-{\theta}_{\widetilde{\Sigma}}(\zeta){\theta}_{\widetilde{\Sigma}}^*(\zeta)\geq0$ and $I-{\theta}^*_{\widetilde{\Sigma}}(\zeta){\theta}_{\widetilde{\Sigma}}(\zeta)\geq0
  $ a.e. on $\zeta \in \dT,$ one concludes that
  \begin{align}
    I-\theta^*(\zeta)\theta(\zeta) &\geq \theta_{21}^*(\zeta)\theta_{21}(\zeta) \label{boundary1}; \\
     I-\theta(\zeta)\theta^*(\zeta) &\geq  \theta_{12}(\zeta)\theta_{12}^*(\zeta) \label{boundary2}.
  \end{align}
Since  $\varphi_{\theta}\equiv0,$ it follows from the identity \eqref{boundary1} and Theorem \ref{majorant}  that $\theta_{21} \equiv 0.$      Then   $D_{T_{,2}}^*=0$ and  $D_{T_{,1}}^*(I-zA)^{-1}B =0$ for every $z$ in some neighbourhood of the origin. 
 Since $\Sigma$ is controllable, it follows from \eqref{cont2}   that $D_{T_{,1}}^*=0$ and then $D_T=0,$ which means that    $\Sigma$ is isometric.

 If $\Sigma$ is chosen to be minimal passive, the previous argument shows that $\Sigma$ is an  isometric and minimal realization of $\theta.$
Since the controllable isometric realizations of $\theta$ are unitarily similar, they all are now also minimal. This proves the first statement in (i).

If  $\theta$ has isometric boundary values a.e., then $\theta_l$ in the  left Kre\u{\i}n-Langer factorization of $\theta$ is inner. Consider  the product $\Sigma=\Sigma_{B^{-1}}\circ  \Sigma_{\theta_l} $ as in the Theorem \ref{reps1}. Let  $\cX_1^{+} \oplus \cX_1^{-}$ and $\cX_2^{+} \oplus \cX_2^{-}$ be the unique fundamental decompositions of $\cX$ of, given by Proposition \ref{invariants}, such that $A\cX_1^{+} \subset \cX_1^{+} $ and  $A^*\cX_2^{+} \subset \cX_2^{+}. $   The case $\kappa=0$ from \cite[Theorem 1.1]{AHS2} shows that the main operator of $\Sigma_{\theta_l}$ belongs to $C_{0\,\cdot},$ and then the   main operator of $\Sigma_{\theta_l}^*$, which is $A^*\uphar_{\cX_2^{+}},$ belongs to $C_{\cdot\,0}.$  It suffices to show that this is equivalent to $A\uphar_{\cX_1^{+}} \in C_{0\,\cdot}. $ Consider   a simple conservative embedding $\widetilde{\Sigma}$ of $\Sigma.$  Represent $\widetilde{\Sigma}$ as in the products 
$\widetilde{\Sigma}=  \Sigma_{\theta_r'} \circ \Sigma_{B^{-1'}_r}=\Sigma_{B_l^{-1'}}\circ  \Sigma_{\theta_l'},$ see Theorem \ref{reps1}. In views of \eqref{product2},  the main operator $A^*\uphar_{\cX_2^{+}}$ of $ \Sigma_{\theta_l'}^*$ belongs to $C_{\cdot\,0},$ and therefore the main operator of $ \Sigma_{\theta_l'}$ belongs to $C_{0\,\cdot},$ see \eqref{product1}.  It follows from Theorem \ref{consinners} that $\theta_l'$ is inner. Then so is $\theta_r',$ and again from the Theorem \ref{consinners} it follows that the main operator $ A\uphar_{\cX_1^{+}}  $ of the system $\Sigma_{\theta_r'}$ is in $C_{0\,\cdot}.$ Then $\Sigma \in  \mathbf{I}^{\kappa}_{0\,\cdot}$, and the second statement in {\rm(i)} is proved.

(ii)
If  $\psi_{\theta}\equiv0,$  the identity \eqref{boundary2} and Theorem \ref{majorant} show that $\theta_{12} \equiv 0,$  which means $D_{T^*_{,2}}=0$ and $C(I-zA)^{-1} D_{T^*_{,1}}\equiv0.$ Since $\Sigma$ is observable, one concludes  as above that  $D_{T^*}=0,$ which means that $\Sigma$ is co-isometric. Similar arguments as above show that $\Sigma$ is also minimal. Moreover, co-isometric boundary values of $\theta$ implies that $\Sigma \in  {\mathbf{I}^{*}}^{\kappa}_{\cdot\,0} .$

 (iii) If $\Sigma$ is simple and $\varphi_{\theta}\equiv0$ and $\psi_{\theta}\equiv0,$ arguments used    in the proof of \cite[Theorem 9.4]{Saprikinetal} show that $\Sigma$ is conservative. Minimality of $\Sigma$ is obtained analogously as above. The last assertion is contained in Theorem \ref{consinners}.
\end{proof}

For the classes   $\mathbf{I}^{\kappa}_{0\,\cdot}$ and  ${\mathbf{I}^{*}}^{\kappa}_{\cdot\,0},$ conditions of Theorem \ref{sepontulos} are also necessary.
\begin{proposition}
An isometric controllable (co-isometric observable) system  $\Sigma$ belongs to  $ \mathbf{I}^{\kappa}_{0\,\cdot}$ (${\mathbf{I}^{*}}^{\kappa}_{\cdot\,0}$)  if and only if $\theta_{\Sigma}$ has isometric (co-isometric) boundary values a.e. on $\dT.$
\end{proposition}
\begin{proof}
  Only the proof of necessity needs to be given. For this, embed $\Sigma$ to  a conservative system $\widetilde{\Sigma}$ with the representation  as in Theorem \ref{reps1} and then apply Theorem \ref{consinners}.
\end{proof}

The existence of a co-isometric observable realization is guaranteed by Theorem \ref{realz}. It is also possible that $\theta \in \SK$ has a co-isometric controllable realization that is neither observable nor conservative.
\begin{example}
  Consider the function in Example \ref{counter} and choose $a$ to be a scalar inner function. Easy calculations show that then $S_l$ is co-inner and the right defect function $\varphi_{S_l} $ of $S_L$ is  not identically zero. Theorem \ref{sepontulos} shows that an
  observable passive realization   $\Sigma_{S_l}$ of $S_l$ is co-isometric and minimal. The property $\varphi_{S_l} \neq 0$ and Lemma \ref{conservs} show that $\Sigma_{S_l}$ cannot be conservative. If $\Sigma_{b^{-1}}$ is a minimal conservative realization of ${b^{-1}},$  Theorem \ref{preserv} shows that  $  \Sigma_{b^{-1}} \circ \Sigma_{S_l}$ is controllable while Example \ref{counter} shows that it is not observable. The product cannot be conservative either,  and thus $S$ has a co-isometric controllable realization.
\end{example}

If the defect functions of $\theta \in \SK$ are zero functions, the results of Theorem \ref{reps1} can be extended.
\begin{proposition} $\Sigma=(A,B,C,D,\cX,\cU,\cY,\kappa)$ be a passive system such that the transfer function $\theta$ of $\Sigma$ belongs to  $\SK.$ Let $\theta=B_l^{-1}\theta_l = \theta_r B_r^{-1}$ be the Kre\u{\i}n--Langer factorizations of $\theta.$ Then the following statements hold:
\begin{itemize}
  \item[{\rm(i)}] If $\varphi_{\theta}\equiv 0, $ then $\Sigma$ can be represented as in the product of the form $$\Sigma=\Sigma_{B_l^{-1}}\circ \Sigma_{\theta_l}, $$ where $\Sigma_{B_l^{-1}}$ and $ \Sigma_{\theta_l}$ and are minimal conservative realization  of $B_l^{-1}$    and passive realization of $\theta_l,$ respectively;
  \item[{\rm(ii)}] If $\psi_{\theta}\equiv 0, $  then $\Sigma$ can be represented as in the product of the form $$\Sigma= \Sigma_{\theta_r}  \circ \Sigma_{B_r^{-1}} $$ where $\Sigma_{B_r^{-1}}$ and $ \Sigma_{\theta_r}$ are minimal conservative realization  of $B_r^{-1}$    and passive realization of $\theta_l,$ respectively;

  \item[{\rm(iii)}] If $\varphi_{\theta}\equiv 0 $  and $\psi_{\theta}\equiv 0,$ then $\Sigma$ can be represented as in the products of the form $$\Sigma=\Sigma_{B_l^{-1}}\circ \Sigma_{\theta_l} =\Sigma_{\theta_r}  \circ \Sigma_{B_r^{-1}}, $$ where $\Sigma_{B_l^{-1}}$ and $\Sigma_{B_r^{-1}}$  are minimal conservative realizations of   $B_l^{-1}$ and $ B_r^{-1},$ respectively, and $\Sigma_{\theta_l}$ and $\Sigma_{\theta_r}$
    are passive realizations of  $ \theta_l$ and $ \theta_r,$ respectively.
\end{itemize}
\end{proposition}
\begin{proof}
Only the proof of (ii) is provided, since the other assertions are obtained analogously.  Suppose that $\psi_{\theta}\equiv 0.$ Lemma \ref{simp-kar} shows that the space $(\cX^c)^\perp$ is a Hilbert space. It follows easily from the identity \eqref{cont1} that $A^* (\cX^c)^\perp\subset (\cX^c)^\perp$ and $B^*(\cX^c)^\perp =\{ 0 \}.$ This implies that the  system operator   can be represented as \begin{equation}\label{projection3}
   T_{\Sigma} =\begin{pmatrix}
                   A_1 & 0 & 0 \\
                  A_2 & A_0 & B_0 \\
                  C_1 & C_0 & D
                \end{pmatrix}: \begin{pmatrix}
                    (\cX^c)^\perp \\ { \cX}^c\\
                    \cU
                  \end{pmatrix}  \to \begin{pmatrix}
                    (\cX^c)^\perp \\ { \cX}^c\\
                    \cY
                  \end{pmatrix}. \end{equation}
 Now easy calculations show that a restriction $\Sigma_0=(A_0,B_0,C_0,D,\cX^c,\cU,\cY,\kappa)$ of $\Sigma$ is controllable and passive, and then according to Theorem \ref{sepontulos}, $\Sigma_0$ is isometric and minimal.  From Theorem \ref{reps1} it follows that $\Sigma_0=\Sigma_{ B_l^{-1}}\circ \Sigma_{ \theta_l} $ and the components have properties introduced in  Theorem \ref{reps1} {\rm(iii)}. The state space $\cX^{c-}$ of $\Sigma_{B_l^{-1}}$ is invariant respect to $A_0.$ Denote the state space of $\Sigma_{ \theta_l}$  by $\cX^{c+}.$     Then $\left(  (\cX^c)^\perp   \oplus \cX^{c+}  \right) \oplus \cX^{c-}$ is a fundamental decomposition of $\cX,$ and
$A \cX^{c-} \subset\cX^{c-}.$  Similar calculations as in the Step 1 (ii) of the  proof of Theorem \ref{reps1} show that
$$  \Sigma= \dil\, \Sigma_0=\dil\left( \Sigma_{ B_l^{-1}}\circ \Sigma_{ \theta_l}\right)  =  \Sigma_{ B_l^{-1}}\circ  \dil\, \Sigma_{ \theta_l},  $$
and this is the desired representation.
\end{proof}

\noindent \textbf{Acknowledgements}\quad This paper will be a part of my forthcoming doctoral thesis. I wish to thank my supervisor Seppo Hassi for helpfull discussion  while preparing this paper.


\begin{thebibliography}{AHS} 
\bibitem{ADRS} 
D.~Alpay, A.~Dijksma, J.~Rovnyak, and H. S. V.~de~Snoo, \emph{Schur
functions, operator colligations, and Pontryagin spaces}, Oper.
Theory Adv. Appl., {96}, Birkh\"auser Verlag, Basel-Boston, 1997.

\bibitem{Ando}
T.~Ando, \emph{De Branges spaces and analytic operator functions}, Division of Applied Mathematics, Research Institute of
Applied Electricity, Hokkaido University, Sapporo, Japan, 1990.












\bibitem{AHS2}
 Yu. M.~Arlinski\u{\i}, S.~Hassi and  H.S.V~de Snoo,
Parametrization of contractive block operator matrices and passive discrete-time systems,
Complex Anal. Oper. Theory 1 (2007), no. 2, 211--233.

\bibitem{Aron}
N.~Aronszajn, Theory of reproducing kernels, Trans. Amer. Math. Soc.
(1950) 68, 337--404.


\bibitem{A}
 D. Z.~Arov, Passive linear steady-state dynamical systems, Sibirsk. Mat. Zh. 20 (1979), no. 2, 211--228  (Russian); English transl. in  Siberian Math. J. 20 (1979), no. 2, 149--162.

\bibitem{Arov}
 D. Z.~Arov, Stable dissipative linear stationary dynamical scattering systems,  J. Operator Theory 2 (1979), no. 1, 95--126 (Russian);  English transl. in  Oper. Theory Adv. Appl., 134, \emph{Interpolation theory, systems theory and related topics (Tel Aviv/Rehovot, 1999)}, 99--136, Birkh{\"a}user, Basel, 2002.

\bibitem{ArKaaP}
D. Z.~Arov, M. A.~Kaashoek, and D. P.~Pik, Minimal and optimal linear
discrete time-invariant dissipative scattering systems. Integr.
Equat. Oper. Theory, 29 (1997), 127--154.



\bibitem{ArKaaP3}
D. Z.~Arov, M. A.~Kaashoek, and D. P.~Pik, The Kalman-Yakubovich-Popov inequality for discrete time systems of infinite dimension, J. Operator Theory 55 (2006), no. 2, 393--438.


\bibitem{ArNu1}
 D. Z.~Arov and M. A.~Nudel'man, A criterion for the unitary similarity of minimal passive systems of scattering with a given transfer function,  Ukra{\"i}n. Mat. Zh. 52 (2000), no. 2, 147--156 (Russian); English
transl. in Ukrainian Math. J. 52 (2000), no. 2, 161--172.

\bibitem{ArNu2}
 D.Z.~Arov and M. A.~Nudel'man, Conditions for the similarity of all minimal passive realizations of a given transfer function (scattering and resistance matrices),   Mat. Sb. 193 (2002), no. 6, 3--24 (Russian); English
transl. in Sb. Math. 193 (2002), no. 5-6, 791--810.

\bibitem{Saprikinetal}
 D. Z.~Arov, J.~Rovnyak and S. M.~Saprikin,
   Linear passive stationary scattering systems with Pontryagin state spaces, Math. Nachr. 279 (2006), no. 13--14, 1396--1424.

\bibitem{SapAr}
D. Z.~Arov and S. M.~Saprikin,  Maximal solutions for embedding problem for a generalized Shur function and optimal dissipative scattering systems with Pontryagin state spaces,
  Methods Funct. Anal. Topology 7 (2001), no. 4, 69--80.

\bibitem{ArSt}
D. Z.~Arov and O. J.~Staffans,
Bi-inner dilations and bi-stable passive scattering realizations of Schur class operator-valued functions,
Integral Equations Operator Theory 62 (2008), no. 1, 29--42.


\bibitem{Azizov}
T. Ya.~Azizov and I. S.~Iokhvidov,
\emph{Foundations of the theory of linear operators in spaces with indefinite metric}, Nauka, Moscow, 1986; English
transl., John Wiley \& Sons Ltd., Chichester, 1989.

\bibitem{BallCohen}
 J. A.~Ball and N.~Cohen,
de Branges-Rovnyak operator models and systems theory: a survey,
\emph{Topics in matrix and operator theory (Rotterdam, 1989)}, 93--136,
Oper. Theory Adv. Appl., 50, Birkh{\"a}user, Basel, 1991.

\bibitem{BGKR}
H.~Bart, , I. Z.~Gohberg, M. A.~Kaashoek and A. C. M.~Ran,
\emph{Factorization of matrix and operator functions: the state space method},
Oper. Theory Adv. Appl., 178, Linear Operators and Linear Systems, Birkh{\"a}user Verlag, Basel, 2008.

\bibitem{Bognar}
 J.~Bogn{\'a}r,
\emph{Indefinite inner product spaces},
Ergebnisse der Mathematik und ihrer Grenzgebiete, Band 78. Springer-Verlag, New York-Heidelberg, 1974.

\bibitem{BrR1}
L.~de Branges and J.~Rovnyak, \emph{Square Summable Power Series}, Holt,
Rinehart and Winston, New-York, 1966.

\bibitem{BrR2}
L.~de Branges and J.~Rovnyak, Appendix on square summable power
series, Canonical models in quantum scattering theory, \emph{Perturbation Theory and its Applications in Quantum Mechanics (Proc. Adv. Sem. Math. Res. Center, U.S. Army, Theoret. Chem. Inst., Univ. of Wisconsin, Madison, Wis., 1965)}, pp. 295--392, Wiley, New York, 1966.


\bibitem{Br1}
M.S.~Brodski\u{\i}, Unitary operator colligations and their characteristic functions, Uspekhi Mat. Nauk 33 (1978), no. 4(202), 141--168, 256  (Russian); English transl. in  Russian Math. Surveys 33 (1978), no. 4,
159--191.


\bibitem{DLS1} 
 A.~Dijksma,  H.~Langer and H. S. V.~de Snoo,
Characteristic functions of unitary operator colligations in $\pi_{\kappa}$-spaces,
\emph{ Operator theory and systems (Amsterdam, 1985)}, 125--194,
  Oper. Theory Adv. Appl., 19, Birkh{\"a}user, Basel, 1986.

\bibitem{DLS2} 
A.~Dijksma,  H.~Langer and H. S. V.~de Snoo,  Unitary colligations in $\Pi_{\kappa}$-spaces,
  characteristic functions and  {\^S}traus extensions,
  Pacific J. Math. 125 (1986), no. 2, 347--362.


\bibitem{rovdrit}
M. A.~Dritschel and J.~Rovnyak,
Operators on indefinite inner product spaces, \emph{Lectures on operator theory and its applications (Waterloo, ON, 1994)}, 141--232,
Fields Inst. Monogr., 3, Amer. Math. Soc., Providence, RI, 1996.








\bibitem{Helton}
J. W.~Helton,
Discrete time systems, operator models, and scattering theory,
J. Functional Analysis 16 (1974), 15--38.





\bibitem{Khanh}
D. C.~Khanh,
($\pm$)-regular factorization of transfer functions and passive scattering systems for cascade coupling,
J. Operator Theory 32 (1994), no. 1, 1--16.


\bibitem{Krein-Langer} 
 M. G.~Kre\u{\i}n and H.~Langer, {\"U}ber die verallgemeinerten Resolventen und die charakteristische Funktion eines isometrischen Operators im Raume $\Pi_\kappa$ (German),\emph{ Hilbert space operators and operator algebras (Proc. Internat. Conf., Tihany, 1970)}, pp. 353--399, Colloq. Math. Soc. J{\'a}nos Bolyai, 5. North-Holland, Amsterdam, 1972.



\bibitem{LangSor}
H.~Langer and P.~Sorjonen, Verallgemeinerte Resolventen hermitescher und isometrischer Operatoren im Pontrjaginraum,
Ann. Acad. Sci. Fenn. Ser. A I No. 561, 1974 (German).



\bibitem{Saprikin1}
S. M.~Saprikin,
 The theory of linear discrete time-invariant dissipative scattering systems with state $\pi_\kappa$-spaces,
 Zap. Nauchn. Sem. S.-Peterburg. Otdel. Mat. Inst. Steklov. (POMI) 282 (2001),
 Issled. po Line\u{\i}n. Oper. i Teor. Funkts. 29, 192--215, 281 (Russian);
 English transl. in J. Math. Sci. (N. Y.) 120 (2004), no. 5, 1752--1765.

\bibitem{Schwartz}
L.~Schwartz, Sous-espaces hilbertiens d'espaces vectoriels topologiques et noyaux associ\'{e}s (noyaux reproduisants), J. Analyse Math. 13 (1964), 115–-256  (French).


\bibitem{Sorjonen} P. Sorjonen, Pontrjaginr{\"a}ume mit einem reproduzierenden Kern, Ann. Acad. Sci. Fenn. Ser. A I Math. No. 594, 1975.

\bibitem{Staffans}
O. J.~Staffans,
 \emph{Well-posed linear systems}, Encyclopedia of Mathematics and its Applications, 103, Cambridge University Press, Cambridge, 2005.

\bibitem{SF}
B.~Sz.-Nagy and C.~Foias, \emph{{Harmonic analysis of operators on
Hilbert space}}, North-Holland, New York, 1970.




\end{thebibliography}
\end{document}